\documentclass[11pt, a4paper]{article} 
\usepackage[latin1]{inputenc}
\usepackage{amsmath}
\usepackage{amsfonts}
\usepackage{amssymb}
\usepackage{amsthm}
\usepackage{amsxtra}
\usepackage{enumerate}

\usepackage{titlesec}
\titleformat{\subsection}[hang]{\normalfont\bfseries}{\thesubsection}{1em}{}

\usepackage{titlesec}

\titlespacing\section{0pt}{3.5ex plus 0.5ex minus .2ex}{0.3ex plus .2ex}
\titlespacing\subsection{0pt}{2.5ex plus 0.5ex minus .2ex}{0.3ex plus .2ex}
\titlespacing\subsubsection{0pt}{2.5ex plus 0.5ex minus .2ex}{0.3ex plus .2ex}

\usepackage{hyperref} 
\usepackage{fullpage}
\usepackage{setspace}

\usepackage{mathrsfs}
\usepackage{url}
\usepackage[all]{xy}
\usepackage{comment}
\usepackage{marginnote}
\usepackage{color}

\usepackage{mathtools}

\numberwithin{equation}{subsection}

\newtheorem{theorem}[equation]{Theorem}
\newtheorem*{theorem*}{Theorem}
\newtheorem{lemma}[equation]{Lemma}
\newtheorem*{lemma*}{Lemma}
\newtheorem{proposition}[equation]{Proposition}
\newtheorem*{proposition*}{Proposition}
\newtheorem{corollary}[equation]{Corollary}
\newtheorem*{corollary*}{Corollary}
\newtheorem{conjecture}[equation]{Conjecture}
\newtheorem*{conjecture*}{Conjecture}

\theoremstyle{definition}
\newtheorem{definition}[equation]{Definition}
\newtheorem*{definition*}{Definition}

\newtheorem*{notation*}{Notation}

\newtheorem*{choice*}{Choice}
\newtheorem{remark}[equation]{Remark}
\newtheorem*{remark*}{Remark}

\newtheorem*{axiom*}{Axiom}

\allowdisplaybreaks[1]

\newcommand{\abs}[1]{\left\lvert#1\right\rvert}

\DeclareMathOperator{\End}{End}

\DeclareMathOperator{\ad}{ad}

\DeclareMathOperator{\ind}{ind}

\DeclareMathOperator{\Hom}{Hom}

\DeclareMathOperator{\red}{red}
\DeclareMathOperator{\Mod}{Mod-}
\DeclareMathOperator{\Res}{Res}

\DeclareMathOperator{\id}{id}

\DeclareMathOperator{\ext}{ext}
\DeclareMathOperator{\aff}{aff}

\DeclareMathOperator{\Gal}{Gal}

\DeclareMathOperator{\normal}{norm}

\DeclareMathOperator{\GL}{GL}
\DeclareMathOperator{\SL}{SL}

\DeclareMathOperator{\Uch}{Uch}

\DeclareMathOperator{\dc}{dc}
\DeclareMathOperator{\qs}{qs}
\DeclareMathOperator{\dual}{d}

\newcommand{\unr}{{\textup{unr}}}

\newcommand{\gen}{{\textup{gen}}}

\newcommand{\Krel}{{\mathcal{K}\textup{-rel}}}
\newcommand{\Kthetarel}{{\mathcal{K}_{\theta}\textup{-rel}}}

\newcommand{\isoarrow}{\stackrel{\sim}{\longrightarrow}}

\newcommand{\Nheart}{N(\rho_{M})^{\heartsuit}_{[x_{0}]_{M}}}

\newcommand{\Wheart}{W(\rho_{M})^{\heartsuit}_{[x_{0}]_{M}}}
\newcommand{\Wthetaheart}{W(u_{\rho_{M}})^{\heartsuit}_{[x_{0}]_{M}}}

\newcommand{\Waff}{W(\rho_M)_{\mathrm{aff}}}
\newcommand{\Wthetaaff}{W(u_{\rho_M})_{\mathrm{aff}}}
\newcommand{\Waffz}{W(\rho_{M^0})_{\mathrm{aff}}}
\newcommand{\Wthetaaffz}{W(u_{\rho_{M^0}})_{\mathrm{aff}}}
\newcommand{\Wzero}{\Omega(\rho_{M})}
\newcommand{\Wzeroz}{\Omega(\rho_{M^0})}
\newcommand{\sfG}{\mathsf{G}}   
\newcommand{\sfM}{\mathsf{M}}   
\newcommand{\sfH}{\mathsf{H}}   
\newcommand{\sfP}{\mathsf{P}}   
\newcommand{\sfT}{\mathsf{T}}   
\newcommand{\Coeff}{\cC}   

\DeclareMathOperator{\Rep}{Rep}
\newcommand{\IEC}{\mathfrak{I}} 

\newcommand{\bC}{\mathbb{C}}

\newcommand{\bQ}{\mathbb{Q}}
\newcommand{\bR}{\mathbb{R}}

\newcommand{\bZ}{\mathbb{Z}}

\newcommand{\cA}{\mathcal{A}}
\newcommand{\cB}{\mathcal{B}}
\newcommand{\cC}{\mathcal{C}}

\newcommand{\cH}{\mathcal{H}}

\newcommand{\cK}{\mathcal{K}}

\newcommand{\cO}{\mathcal{O}}

\newcommand{\ff}{\mathfrak{f}}

\newcommand{\fs}{{\mathfrak{s}}}

\newcommand{\fS}{\mathfrak{S}}

\begin{document}
\author{Kazuma Ohara}

\AtEndDocument{\bigskip{\footnotesize%
	\par
	\textsc{%
	Max Planck Institute for Mathematics, Vivatsgasse 7, 53115 Bonn, Germany}
	\par
       \textit{E-mail address}: \texttt{kazuma@mpim-bonn.mpg.de}
}}

\title{On parameters of Hecke algebras for $p$-adic groups} 
\date{}
\maketitle
\begin{abstract}
Let $F$ be a non-archimedean local field with residue characteristic $p$ and $G$ be a connected reductive group defined over $F$.
In earlier joint works with Jeffrey D. Adler, Jessica Fintzen, and Manish Mishra, we proved that the Hecke algebras attached to types constructed by Kim and Yu are isomorphic to the Hecke algebras attached to depth-zero types.
Note that if $G$ splits over a tamely ramified extension of $F$ and $p$ does not divide the order of the absolute Weyl group of $G$, such Hecke algebras cover the Hecke algebras attached to arbitrary Bernstein blocks.
We also proved that for a depth-zero type $(K, \rho)$, the corresponding Hecke algebra $\cH(G(F), (K, \rho))$ has an explicit description as a semi-direct product of an affine Hecke algebra $\mathcal{H}(W(\rho_M)_{\mathrm{aff}}, q)$ with a twisted group algebra $\mathbb{C}[\Omega(\rho_{M}), \mu]$, generalizing prior work of Morris.

In this paper, we show that the affine Hecke algebra $\mathcal{H}(W(\rho_M)_{\mathrm{aff}}, q)$ appearing in the description of the Hecke algebra $\mathcal{H}(G(F), (K, \rho))$ attached to a depth-zero type $(K, \rho)$ is isomorphic to the one attached to a unipotent type for a connected reductive group splitting over an unramified extension of $F$.
This makes it possible to calculate the parameters of the affine Hecke algebras for depth-zero types and types constructed by Kim and Yu explicitly.
In particular, we prove a version of Lusztig's conjecture that the parameters of the Hecke algebra attached to an arbitrary Bernstein block agree with those of a unipotent Bernstein block under the assumption that $G$ splits over a tamely ramified extension of $F$ and $p$ does not divide the order of the absolute Weyl group of $G$.
\end{abstract}

{
	\renewcommand{\thefootnote}{}  
	\footnotetext{MSC2020: 22E50, 20C08, 20C33}
	\footnotetext{Keywords: $p$-adic groups, smooth representations, Hecke algebras, types, unipotent representations}
}

\setcounter{tocdepth}{2}

\tableofcontents

\numberwithin{equation}{section}
\section{Introduction}

The category of all smooth, complex representations of a $p$-adic group decomposes as a product of
indecomposable,
full subcategories, called Bernstein blocks, each of which is equivalent to modules over a Hecke algebra under minor tameness assumptions. Therefore knowing the explicit structure of these Hecke algebras and their modules yields an understanding of the category of smooth representations.
In \cite{2024arXiv240807805A, 2024arXiv240807801A}, we showed that under the above minor tameness assumptions, the Hecke algebra attached to an arbitrary Bernstein block is isomorphic to a depth-zero Hecke algebra and has an explicit description as a semi-direct product of an affine Hecke algebra with a twisted group algebra, generalizing prior work of Morris (\cite{Morris}).
In this paper, we further prove that the affine Hecke algebra appearing in the description above is isomorphic to the affine Hecke algebra appearing in the similar description of the Hecke algebra attached to a Bernstein block consisting of \emph{unipotent} representations.

We explain our results in more detail.
Let $F$ be a non-archimedean local field with residue field $\ff$ and $G$ be a connected reductive group defined over $F$.
According to the \emph{Bernstein decomposition} (\cite{MR771671}), the category $\Rep(G(F))$ of smooth, complex representations of $G(F)$ decomposes into a product of indecomposable, full subcategories, $\Rep^\fs(G(F))$, that are called \emph{Bernstein blocks} and that are indexed by the set of inertial equivalence classes $\IEC(G)$, i.e., equivalence classes $[L, \sigma]_{G}$ of pairs $(L, \sigma)$ consisting of a Levi subgroup $L$ of (a parabolic subgroup of) $G$ and an irreducible supercuspidal representation $\sigma$ of $L(F)$:
\[
\Rep(G(F))=\prod_{\fs\in\IEC(G)}\Rep^\fs(G(F)).
\]
A pair $(K, \rho)$ consisting of a compact, open subgroup $K \subseteq G(F)$ and an irreducible smooth representation $(\rho, V_\rho)$ of $K$ is called an $\fs$-type for $\fs \in \IEC(G)$
if $\Rep^\fs(G(F))$ contains precisely those smooth representations $\pi$ of $G(F)$
such that every irreducible subquotient of $\pi$ contains $\rho$ upon restriction to $K$.
If $(K, \rho)$ is an $\fs$-type, then we have an equivalence of categories between $\Rep^{\fs}(G(F))$ and the category of right unital $\cH(G(F), (K, \rho))$-modules:
\[
\Rep^{\fs}(G(F)) \simeq \Mod \cH(G(F), (K, \rho)),
\]
where  $\cH(G(F),(K, \rho))$ denotes the Hecke algebra attached to $(K, \rho)$, i.e., as a vector space the collection
of all compactly supported,
$\End_{\bC}(V_\rho)$-valued
functions on $G(F)$ that transform
on the left and right under $K$  
by $\rho$.
The algebra structure on $\cH(G(F),(K, \rho))$ is given by convolution,
see \cite[Section~2.2]{2024arXiv240807801A}
for details.
For some special classes of groups, explicit constructions of types and descriptions of their Hecke algebras were given, for instance, Bushnell and Kutzko (\cite{MR1204652}) for the general linear groups $\GL_{n}$, Goldberg and Roche 
(\cite{MR1901371,GoldbergRoche-Hecke}) for the special linear groups $\SL_{n}$, S\'echerre and Stevens (\cite{secherre-stevens:glnd-4}) for all inner forms of $\GL_n$, Miyauchi and Stevens
(\cite{MR3157998}) for types associated to maximal proper Levi subgroups of classical groups,
for more details, see \cite[Section~1.2]{2024arXiv240807801A}.

On the other hand, for general connected reductive groups but restricting the class of representations to \emph{depth-zero representations}, Moy and Prasad (\cite{MR1253198, MR1371680}) and, independently, Morris (\cite{MR1713308}) provided a construction of types called \emph{depth-zero types}.
A depth zero type $(K, \rho)$ consists of a compact, open subgroup $K$ of $G(F)$ containing a parahoric subgroup $G(F)_{x, 0}$ of $G(F)$ as a normal subgroup and an irreducible representation $\rho$ of $K$ whose restriction to $G(F)_{x, 0}$ contains the inflation of an irreducible, cuspidal representation $\rho_{0}$ of $\sfG^{\circ}_{x}(\ff)$, where $\sfG^{\circ}_{x}$ denotes the reductive quotient of the special fiber of the connected parahoric group scheme associated to $G(F)_{x, 0}$.
Building upon the construction of depth-zero types and the construction of supercuspidal representations by Yu (\cite{Yu}) and using the theory of covers introduced by Bushnell and Kutzko (\cite{BK-types}), Kim and Yu (\cite{Kim-Yu, MR4357723}) provided a construction of types for a connected reductive group that splits over a tamely ramified extension of $F$. This construction yields types for every Bernstein block if the characteristic $p$ of $\ff$ does not divide the order $\abs{W}$ of the absolute Weyl group $W$ of $G$ by \cite{Fi-exhaustion}. Thus, understanding the structure of the corresponding Hecke algebras and their categories of modules leads to an understanding of the whole category of smooth representations of $G(F)$ if $G$ splits over a tamely ramified extension of $F$ and $p \nmid \abs{W}$.
In \cite{2024arXiv240807805A, 2024arXiv240807801A}, we showed that the Hecke algebra attached to an arbitrary type constructed by Kim and Yu is isomorphic to the Hecke algebra attached to a depth-zero type.
Moreover, we generalized \cite[7.12~Theorem]{Morris} and 
described the structure of the Hecke algebra $\cH(G(F), (K, \rho))$ attached to a depth-zero type $(K, \rho)$ via an explicit isomorphism 
\begin{equation}
\label{introheckealgebraisomrecall}
\cH(G(F), (K, \rho)) \isoarrow \bC[\Wzero, \mu] \ltimes \cH(\Waff, q),
\end{equation}
where $\Waff$ is an affine Weyl group
that is a normal subgroup of a larger symmetry group $\Wzero \ltimes \Waff$,
$q \colon S \longrightarrow \bC^\times$
is a parameter function on a set $S$ of simple reflections generating $\Waff$,
$\cH(\Waff, q)$ denotes the corresponding affine Hecke algebra,
and
$\bC[\Wzero, \mu]$ denotes the group algebra of $\Wzero$
twisted by a $2$-cocycle $\mu$.
In \cite[Proposition~3.9.1]{2024arXiv240807801A}, which is based on \cite[Theorem~3.18 (ii)]{MR570873}, we provided a way to calculate the parameter function $q$ using a calculation involving representations of finite reductive groups.
On the other hand, an explicit description of the parameter function $q$ has not been obtained.

In this paper, we prove that the affine Hecke algebra $\cH(\Waff, q)$ in \eqref{introheckealgebraisomrecall} is isomorphic to the affine Hecke algebra $\cH(\Wthetaaff, q_{\theta})$ appearing in the explicit isomorphism
\[
\cH(G_{\theta}(F), (K_{\theta}, u)) \isoarrow \bC[\Omega(u_{\rho_{M}}), \mu_{\theta}] \ltimes \cH(\Wthetaaff, q_{\theta})
\]
for a \emph{unipotent type} $(K_{\theta}, u)$ for a connected reductive group $G_{\theta}$ that splits over an unramified extension of $F$  (Theorem~\ref{mainthm:heckealgebraisom}).
Here, a unipotent type means a depth-zero type $(K, \rho)$ such that $\rho_{0}$ is a unipotent representation of $\sfG^{\circ}_{x}(\ff)$.
Similarly, an irreducible smooth representation $\pi$ of $G(F)$ whose restriction to a parahoric subgroup $G(F)_{x, 0}$ contains the inflation of a unipotent, cuspidal representation of $\sfG^{\circ}_{x}(\ff)$ is called a \emph{unipotent representation} of $G(F)$.
Unipotent representations are the representations that are expected to correspond to the $L$-parameters that are trivial on the inertia subgroup of the Weil group of $F$ via the local Langlands correspondence.
The structure and representation theory of Hecke algebras attached to unipotent types were well-studied by Lusztig (\cite{Lusztig95unipotent, Lusztig02unipotentII}), and Lusztig used it to classify the unipotent representations of adjoint simple algebraic groups that split over an unramified extension.
Lusztig's classification of unipotent representations provides us an important case of an explicit local Langlands correspondence.
The restriction to unramified adjoint groups was later removed by Solleveld (\cite{Solleveld23unipotentLLC}, \cite{MR4620884}), building on foundational work by many mathematicians, including \cite{MR3845761, MR3969871, 2017arXiv170103593A, MR4167790, MR4515286}.

The main result Theorem~\ref{mainthm:heckealgebraisom} of this paper makes it possible to calculate the parameters of the affine Hecke algebra $\cH(\Waff, q)$ in \eqref{introheckealgebraisomrecall} since the parameters of a unipotent Hecke algebra were explicitly described by Lusztig in \cite{Lusztig95unipotent, Lusztig02unipotentII} (see also \cite[Table\,1]{MR4847675}).
This reduction of the calculation of parameters to the unipotent case was conjectured by Lusztig in \cite{2020arXiv200608535L}.

Moreover, Theorem~\ref{mainthm:heckealgebraisom} can also be regarded as the ``affine Hecke algebra part'' of a reduction to unipotent Hecke algebra isomorphism.
Since the Bernstein blocks consisting of unipotent representations are much easier than general Bernstein blocks, constructing an isomorphism between the Hecke algebra for an arbitrary Bernstein block and the Hecke algebra for a unipotent Bernstein block is an important problem.
In the case of $\GL_{n}$, the Hecke algebras for arbitrary Bernstein blocks are isomorphic to the Hecke algebras for unipotent Bernstein blocks by the work of Bushnell and Kutzko (\cite{MR1204652}).
This reduction to unipotent result is an important tool for a variety of applications, including its recent use in the construction of a part of a categorical local Langlands correspondence (\cite{BZCHN}). 
Moreover, for more general reductive groups, the reduction to unipotent result was previously obtained by Roche (\cite{MR1621409}) in the case where $F$ has characteristic zero, the group $G$ splits over $F$, the characteristic $p$ of $\ff$ is not too small, and the type $(K, \rho)$ is for a Bernstein block consisting of principal series representations.
Roche proved that the Hecke algebra $\cH(G(F), (K, \rho))$ is isomorphic to the Iwahori-Hecke algebra $\cH(\widetilde{H}, 1)$ of a disconnected reductive groups $\widetilde{H}$ under these assumptions (\cite[Theorem~8.2]{MR1621409}).
The main result of this paper is a generalization of the restriction of Roche's isomorphism $\cH(G(F), (K, \rho)) \isoarrow \cH(\widetilde{H}, 1)$ to the affine Hecke algebra part $\cH(\Waff, q)$ of $\cH(G(F), (K, \rho))$.

\subsection*{Sketch of the proof of the main theorem and the structure of the paper}

For a finite-dimensional, length-two representation $\pi$ of a group $G$, we define the $q$-parameter $q(\pi)$ attached to $\pi$ by $q(\pi) = \frac{\dim(\pi_{1})}{\dim(\pi_{2})}$, where $\pi_{1}$ and $\pi_{2}$ are the irreducible constitutes of $\pi$ such that $\dim(\pi_{1}) \ge \dim(\pi_{2})$.
Then the parameters of the affine Hecke algebra $\cH(\Waff, q)$ in \eqref{introheckealgebraisomrecall} agree with $q(\ind_{K_1}^{K_2} (\rho))$ for suitable compact, open subgroups $K_1 \subseteq K_2 \subseteq G(F)$ and irreducible smooth representations $\rho$ of $K_{1}$ (see Proposition~\ref{calculationofqs}).

In Section~\ref{section:comparisonofqparameter}, we prove that the $q$-parameters attached to compactly induced representations are preserved by certain group homomorphisms.

In Section~\ref{section:finitegroupcase}, we study the $q$-parameter $q(R_{\sfM}^{\sfG} (\rho_{\sfM}))$ attached to the parabolic induction $R_{\sfM}^{\sfG} (\rho_{\sfM})$ of an irreducible, cuspidal representation $\rho_{\sfM}$ of a maximal Levi subgroup $\sfM$ of a connected reductive group $\sfG$ defined over a finite field $\ff$.
We prove in Proposition~\ref{prop:reductiontounipotentqparameters} and Corollary~\ref{corollary:reductiontounipdualver} that the $q$-parameter $q(R_{\sfM}^{\sfG} (\rho_{\sfM}))$ agrees with the $q$-parameter attached to the parabolic induction of a unipotent, cuspidal representation.
The key ingredients here are the reduction arguments using the result in Section~\ref{section:comparisonofqparameter} and Jordan decompositions for Lusztig series.

We prove the main theorem of this paper in Section~\ref{section:heckealgebraisom}.
In $\S$\ref{subsec:reviewofAFMO}, we recall the description of the Hecke algebra $\cH(G(F), (K_{x_0}, \rho_{x_0}))$ attached to a depth-zero type $(K_{x_0}, \rho_{x_0})$ given in \cite[Section~5]{2024arXiv240807801A}.
In $\S$\ref{subsec:defofGthetaandu}, we define a connected reductive subgroup $G_{\theta}$ that splits over an unramified extension of $F$ and a unipotent type $(K_{\theta, x_0}, \rho_{\theta, x_0})$ for $G_{\theta}$ from the depth-zero type $(K_{x_0}, \rho_{x_0})$.
In Theorem~\ref{mainthm:heckealgebraisom}, we prove that the affine Hecke algebras $\cH(\Waff, q)$ and $\cH(\Wthetaaff, q_{\theta})$ appearing in the descriptions of $\cH(G(F), (K_{x_0}, \rho_{x_0}))$ and $\cH\left(
G_{\theta}(F), (K_{\theta, x_{0}}, u_{x_{0}})
\right)$ are isomorphic.
To prove Theorem~\ref{mainthm:heckealgebraisom}, we use the reduction argument from Section~\ref{section:comparisonofqparameter} and the comparison of $q$-parameters from Section~\ref{section:finitegroupcase}.
In $\S$\ref{subsec:heckealgebraforkimyutypes}, we combine Theorem~\ref{mainthm:heckealgebraisom} with \cite[Theorem~4.4.1]{2024arXiv240807805A} to prove that the affine Hecke algebra appearing in the description of the Hecke algebra for an arbitrary type constructed by Kim and Yu is isomorphic to the affine Hecke algebra appearing in the description of a unipotent Hecke algebra.
In $\S$\ref{subsection:Lusztig'sconjecture}, we prove a version of Lusztig's conjecture about parameters of Hecke algebras assuming that $G$ splits over a tamely ramified extension of $F$ and $p \nmid \abs{W}$.

\subsection*{Acknowledgments}

The author would like to thank Jessica Fintzen and Jeffrey D. Adler for their valuable feedback on an earlier draft of this paper. He is also grateful to Noriyuki Abe and David Schwein for helpful discussions on Section~\ref{subsec:defofGthetaandu}, and to Arnaud Eteve for sharing his results on categorical Jordan decompositions.
He thanks Maarten Solleveld for his helpful comments on the manuscript and for pointing him to many related references.
Finally, the author gratefully acknowledges the hospitality and financial support of the Max Planck Institute for Mathematics.

\subsection*{General notation}

We fix a prime number $p$ and $\ell$ such that $p \neq \ell$.
Let $\Coeff$ be a field that is isomorphic to $\overline{\bQ_{\ell}}$.
We fix an isomorphism $\Coeff \simeq \overline{\bQ_{\ell}}$.
In this paper, all representations are on vector spaces over $\Coeff$.

If $L$ is any field, and $L'$ is any Galois field extension of $L$,
then we
let $\Gal(L'/L)$ denote the corresponding Galois group.

For a representation $(\rho, V_{\rho})$ of a group $H$, we identify $\rho$ with its representation space $V_{\rho}$ by abuse of notation. 
For any vector space $V$, we write $\id_V$
for the identity map on $V$.

Suppose that $K$ is a subgroup of a group $H$ and $h \in H$. We denote $hKh^{-1}$ by $^hK$. If $\rho$ is a representation of $K$, we write $^h\!\rho$ for the representation $x\mapsto \rho(h^{-1}xh)$ of $^h\!K$. 
We define
\[
N_{H}(K) = \left\{
h \in H \mid {^h\!K} = K
\right\}
\quad
\text{ and }
\quad
N_{H}(\rho) = \left\{
h \in N_{H}(K) \mid {^h\!\rho} \simeq \rho
\right\}.
\]

Suppose that $K$ is an open subgroup of a locally profinite group $H$.
For a smooth representation $(\rho, V_{\rho})$ of $K$, we denote by 
\[
\left(
\ind_{K}^{H} (\rho), \ind_{K}^{H} (V_{\rho})
\right)
\]
the compactly induced representation of $(\rho, V_{\rho})$.

Let $\pi$ be a finite-dimensional representation of a group $G$, and suppose that $\pi$ has length at most two.
Then we define a number $q(\pi) \ge 1$ as follows.
If $\pi$ is irreducible, we define $q(\pi) = 1$.
Suppose that the representation $\pi$ has length two, and let $\pi_{1}$ and $\pi_{2}$ be the irreducible constitutes of $\pi$ such that $\dim(\pi_{1}) \ge \dim(\pi_{2})$.
Then we define the number $q(\pi)$ as $q(\pi) = \frac{\dim(\pi_{1})}{\dim(\pi_{2})}$.
We call $q(\pi)$ the \emph{$q$-parameter} attached to $\pi$.
We regard $q(\pi) \in \Coeff$ via the inclusion $\bQ \subseteq \Coeff$.

\section{A comparison of $q$-parameters via group homomorphisms}
\label{section:comparisonofqparameter}

In this section, we will prove a comparison result of $q$-parameters attached to compactly induced representations via certain group homomorphisms, which we will use in the reduction arguments in following sections.

Let $\sfG'$ be a finite group.
Let $\sfG$ be a normal subgroup of $\sfG'$ and $\sfP'$ be a subgroup of $\sfG'$ such that $\sfG' = \sfP' \cdot \sfG$.
We write $\sfP = \sfP' \cap \sfG$.
We suppose that $\sfP'/\sfP \simeq \sfG'/\sfG$ is abelian.
Let $\rho'$ be an irreducible representation of $\sfP'$ and $\rho$ be an irreducible subrepresentation of $\rho' \restriction_{\sfP}$.
Then there exists $p'_{1}, p'_{2}, \ldots, p'_{n} \in \sfP'$ such that the restriction $\rho' \restriction_{\sfP}$ decomposes as $\rho' \restriction_{\sfP} = \bigoplus_{1 \le k \le n} {^{p'_{k}}\rho}$.
We suppose that the representations $\ind_{\sfP'}^{\sfG'} (\rho')$ and $\ind_{\sfP}^{\sfG} (\rho)$ have length at most two.
\begin{proposition}
\label{prop:comparisonofqparameters}
We have 
$
q\left(
\ind_{\sfP'}^{\sfG'} (\rho')
\right)
=
q
\left(
\ind_{\sfP}^{\sfG} (\rho)
\right)
$.
\end{proposition}

Proposition~\ref{prop:comparisonofqparameters} will be used in the reduction arguments in Sections~\ref{section:finitegroupcase} and \ref{section:heckealgebraisom}.
In particular, we will apply Proposition~\ref{prop:comparisonofqparameters} to the following cases:
\begin{description}
\item[Reduction to connected center case]
Let $i \colon \sfG \rightarrow \sfG'$ be a regular embedding of connected reductive groups defined over a finite field $\ff$, that is, $i$ induces an isomorphism of the adjoint groups, and the center of the group $\sfG'$ is connected.
We will apply Proposition~\ref{prop:comparisonofqparameters} to $\sfG' = \sfG'(\ff)$ and $\sfG = \sfG(\ff)$ in Section~\ref{subsec:reductiontosimpleconnectedcenter} to reduce the calculation of $q$-parameters to the  connected center case, where the representation theory of finite reductive groups is simpler than the general case.

\item[Reduction to connected parahoric subgroups]
Let $G$ be a connected reductive group defined over a non-archimedean local field $F$ with residue field $\ff$.
In order to construct a type for a single depth-zero Bernstein block of $G$, we use a representation of the stabilizer $G(F)_{x}$ of a point $x$ of the enlarged Bruhat--Tits building  $\cB(G, F)$ that is trivial on the pro-$p$ radical $G(F)_{x, 0+}$ of the parahoric subgroup $G(F)_{x,0}$.
The quotient group $G(F)_{x}/G(F)_{x, 0+}$ can be identified with the group of $\ff$-points of a reductive group $\sfG_{x}$ defined over $\ff$.
However, the group $\sfG_{x}$ is disconnected in general.
We will apply Proposition~\ref{prop:comparisonofqparameters} to $\sfG' = \sfG_{x}(\ff)$ and $\sfG = \sfG_{x}^{\circ}(\ff)$, where $\sfG_{x}^{\circ}$ is the connected component of $\sfG_{x}$, in the proof of Theorem~\ref{mainthm:heckealgebraisom} to reduce the calculations of $q$-parameters attached to representations of $\sfG_{x}(\ff)$ to the similar calculations for the connected reductive group $\sfG^{\circ}_{x}$.

\end{description}

\begin{proof}
\addtocounter{equation}{-1}
\begin{subequations}
First, we consider the case that the lengths of the representations $\ind_{\sfP'}^{\sfG'} (\rho')$ and $\ind_{\sfP}^{\sfG} (\rho)$ are two.
We write $\ind_{\sfP'}^{\sfG'} (\rho') = \rho'_{1} \oplus \rho'_{2}$ and $\ind_{\sfP}^{\sfG} (\rho) = \rho_{1} \oplus \rho_{2}$.
We decompose the restrictions $\rho'_{1} \restriction_{\sfG}$ and $\rho'_{2} \restriction_{\sfG}$ as
\[
\rho'_{1} \restriction_{\sfG} = \bigoplus_{1 \le i \le l} \rho'_{1, i}
\qquad
\text{and}
\qquad
\rho'_{2} \restriction_{\sfG} = \bigoplus_{1 \le j \le m} \rho'_{2, j}.
\]
Then we have
\begin{align*}
\left(
\bigoplus_{1 \le i \le l} \rho'_{1, i}
\right)
\oplus 
\left(
\bigoplus_{1 \le j \le m} \rho'_{2, j}
\right)
&=
\rho'_{1} \restriction_{\sfG} \oplus \rho'_{2} \restriction_{\sfG} 
=
\ind_{\sfP'}^{\sfG'} (\rho') \restriction_{\sfG} \\
&=
\ind_{\sfP}^{\sfG} \left(
\rho' \restriction_{\sfP}
\right) 
=
\bigoplus_{1 \le k \le n} \ind_{\sfP}^{\sfG} ({^{p'_{k}}\rho})
= 
\bigoplus_{1 \le k \le n} \left(
{^{p'_{k}}\rho_{1}} \oplus {^{p'_{k}}\rho_{2}}
\right). 
\end{align*}
Hence, as multisets, we have
\begin{equation}
\label{equalityasmultisets}
\left\{
 \rho'_{1, 1},  \rho'_{1, 2}, \ldots,  \rho'_{1, l},  \rho'_{2, 1},  \rho'_{2, 2}, \ldots,  \rho'_{2, m}
\right\}
=
\left\{
{^{p'_{1}}\rho_{1}}, {^{p'_{1}}\rho_{2}}, {^{p'_{2}}\rho_{1}}, {^{p'_{2}}\rho_{2}}, \ldots, {^{p'_{n}}\rho_{1}}, {^{p'_{n}}\rho_{2}}
\right\}.
\end{equation}
Since $\sfG$ is a normal subgroup of $\sfG'$, we have $\dim( \rho'_{1, i}) = \dim( \rho'_{1})/l$ and $\dim( \rho'_{2, j}) = \dim( \rho'_{2})/m$ for all $1 \le i \le l$ and $1 \le j \le m$.
Hence, by comparing the dimensions of the representations in \eqref{equalityasmultisets}, as multisets, we have
\begin{align}
\label{equationofmultisetsdim}
&\left\{
\underbrace{\frac{\dim( \rho'_{1})}{l}, \frac{\dim( \rho'_{1})}{l}, \cdots, \frac{\dim( \rho'_{1})}{l}}_{\text{$l$ times}}, \underbrace{\frac{\dim( \rho'_{2})}{m}, \frac{\dim( \rho'_{2})}{m}, \cdots, \frac{\dim( \rho'_{2})}{m}}_{\text{$m$ times}}
\right\} \\
= 
&\left\{
\underbrace{
\dim(\rho_{1}), \dim(\rho_{1}), \ldots, \dim(\rho_{1})
}_{\text{$n$ times}}, \underbrace{
\dim(\rho_{2}), \dim(\rho_{2}), \ldots, \dim(\rho_{2})
}_{\text{$n$ times}}
\right\}. \notag
\end{align}

Suppose that $\Hom_{\sfG}\left(
\rho'_{1} \restriction_{\sfG}, \rho'_{2} \restriction_{\sfG} 
\right) \neq \{0\}$.
Then according to \cite[Lemma~A.4.3]{2019arXiv191203274K}, we have $ \rho'_{2} \simeq  \rho'_{1} \otimes \chi$ for a character $\chi$ of $\sfG'/\sfG$.
Hence, we obtain that $\rho'_{1} \restriction_{\sfG} \simeq \rho'_{2} \restriction_{\sfG}$.
In particular, we have $\dim( \rho'_{1}) = \dim( \rho'_{2})$ and $l = m$, which implies that all elements of the left hand side of \eqref{equationofmultisetsdim} are the same.
Thus, we obtain that $q\left(
\ind_{\sfP'}^{\sfG'} (\rho')
\right)
=
q
\left(
\ind_{\sfP}^{\sfG} (\rho)
\right)
=1$ in this case.

Now, we consider the case that $\Hom_{\sfG}\left(
\rho'_{1} \restriction_{\sfG}, \rho'_{2} \restriction_{\sfG} 
\right) = \{0\}$, which means that the multisets 
$
\left\{
 \rho'_{1, 1},  \rho'_{1, 2}, \ldots,  \rho'_{1, l}
\right\}
$
and
$\left\{
 \rho'_{2, 1},  \rho'_{2, 2}, \ldots,  \rho'_{2, m}
\right\}$ are disjoint.
By replacing the indices if necessary, we assume that $\rho_{1} \in \left\{
 \rho'_{1, 1},  \rho'_{1, 2}, \ldots,  \rho'_{1, l}
\right\}$.
Then since $\left\{
 \rho'_{1, 1},  \rho'_{1, 2}, \ldots,  \rho'_{1, l}
\right\}$ is the multiset of the irreducible subrepresentations of the restriction of the representation $\rho'_{1}$ of $\sfG'$ to $\sfG$, we have ${^{p'_{k}}\rho_{1}} \in \left\{
 \rho'_{1, 1},  \rho'_{1, 2}, \ldots,  \rho'_{1, l}
\right\}$ for all $1 \le k \le n$.
If ${^{p'_{k}}\rho_{2}} \in \left\{
 \rho'_{1, 1},  \rho'_{1, 2}, \ldots,  \rho'_{1, l}
\right\}$ for some $k$, we have ${^{p'_{k}}\rho_{2}} \in \left\{
 \rho'_{1, 1},  \rho'_{1, 2}, \ldots,  \rho'_{1, l}
\right\}$ for all $1 \le k \le n$, contradicting \eqref{equalityasmultisets} and the fact that the multisets 
$
\left\{
 \rho'_{1, 1},  \rho'_{1, 2}, \ldots,  \rho'_{1, l}
\right\}
$
and
$\left\{
 \rho'_{2, 1},  \rho'_{2, 2}, \ldots,  \rho'_{2, m}
\right\}$ are disjoint.
Hence, we obtain that 
\[
\left\{
 \rho'_{1, 1},  \rho'_{1, 2}, \ldots,  \rho'_{1, l}
\right\}
=
\left\{
{^{p'_{1}}\rho_{1}}, {^{p'_{2}}\rho_{1}}, \ldots, {^{p'_{n}}\rho_{1}}
\right\}
\quad
\text{and}
\quad
\left\{
 \rho'_{2, 1},  \rho'_{2, 2}, \ldots,  \rho'_{2, m}
\right\}
=
\left\{
{^{p'_{1}}\rho_{2}}, {^{p'_{2}}\rho_{2}}, \ldots, {^{p'_{n}}\rho_{2}}
\right\}.
\]
Thus, we obtain that $l = m = n$ and $q\left(
\ind_{\sfP'}^{\sfG'} (\rho')
\right)
=
q
\left(
\ind_{\sfP}^{\sfG} (\rho)
\right)$.

If the representations $\ind_{\sfP'}^{\sfG'} (\rho')$ and $\ind_{\sfP}^{\sfG} (\rho)$ are irreducible, we have $q\left(
\ind_{\sfP'}^{\sfG'} (\rho')
\right)
=
q
\left(
\ind_{\sfP}^{\sfG} (\rho)
\right) = 1$ by definition.
Suppose that $\ind_{\sfP'}^{\sfG'} (\rho')$ is irreducible and $\ind_{\sfP}^{\sfG} (\rho)$ has length two.
By using the same argument as above, we have
\begin{align*}
&\quad \left\{
\underbrace{\frac{\dim\left( \ind_{\sfP'}^{\sfG'} (\rho')\right)}{l}, \frac{\dim\left( \ind_{\sfP'}^{\sfG'} (\rho')\right)}{l}, \cdots, \frac{\dim\left( \ind_{\sfP'}^{\sfG'} (\rho')\right)}{l}}_{\text{$l$ times}}\right\} \\
&= 
\left\{
\underbrace{
\dim(\rho_{1}), \dim(\rho_{1}), \ldots, \dim(\rho_{1})
}_{\text{$n$ times}}, \underbrace{
\dim(\rho_{2}), \dim(\rho_{2}), \ldots, \dim(\rho_{2})
}_{\text{$n$ times}}
\right\},
\end{align*}
where $l$ is the length of $ \ind_{\sfP'}^{\sfG'} (\rho') \restriction_{\sfG}$.
Hence, we have $\dim(\rho_{1}) = \dim(\rho_{2}) = \dim\left( \ind_{\sfP'}^{\sfG'} (\rho')\right)/l$.
Thus, we obtain that $q\left(
\ind_{\sfP'}^{\sfG'} (\rho')
\right)
=
q
\left(
\ind_{\sfP}^{\sfG} (\rho)
\right) = 1$.

Finally, suppose that $\ind_{\sfP'}^{\sfG'} (\rho')$ has length two and $\ind_{\sfP}^{\sfG} (\rho)$ is irreducible.
In this case, we have
\[
\left\{
 \rho'_{1, 1},  \rho'_{1, 2}, \ldots,  \rho'_{1, l},  \rho'_{2, 1},  \rho'_{2, 2}, \ldots,  \rho'_{2, m}
\right\}
=
\left\{
{^{p'_{1}}\!\ind_{\sfP}^{\sfG} (\rho)}, {^{p'_{2}}\!\ind_{\sfP}^{\sfG} (\rho)}, \ldots, {^{p'_{n}}\!\ind_{\sfP}^{\sfG} (\rho)}
\right\}.
\]
Suppose that $\Hom_{\sfG}\left(
\rho'_{1} \restriction_{\sfG}, \rho'_{2} \restriction_{\sfG} 
\right) = \{0\}$.
Then by the same argument as above, we obtain that 
\[
\left\{
 \rho'_{1, 1},  \rho'_{1, 2}, \ldots,  \rho'_{1, l}
\right\}
=
\left\{
{^{p'_{1}}\!\ind_{\sfP}^{\sfG} (\rho)}, {^{p'_{2}}\!\ind_{\sfP}^{\sfG} (\rho)}, \ldots, {^{p'_{n}}\!\ind_{\sfP}^{\sfG} (\rho)}
\right\}
\]
or
\[
\left\{
 \rho'_{2, 1},  \rho'_{2, 2}, \ldots,  \rho'_{2, m}
\right\}
=
\left\{
{^{p'_{1}}\!\ind_{\sfP}^{\sfG} (\rho)}, {^{p'_{2}}\!\ind_{\sfP}^{\sfG} (\rho)}, \ldots, {^{p'_{n}}\!\ind_{\sfP}^{\sfG} (\rho)}
\right\},
\]
a contradiction.
Hence, we obtain that $\Hom_{\sfG}\left(
\rho'_{1} \restriction_{\sfG}, \rho'_{2} \restriction_{\sfG} 
\right) \neq \{0\}$.
In this case, we have $ \rho'_{2} \simeq  \rho'_{1} \otimes \chi$ for a character $\chi$ of $\sfG'/\sfG$ by \cite[Lemma~A.4.3]{2019arXiv191203274K}.
Thus, we obtain that $\dim( \rho'_{1}) = \dim( \rho'_{2})$ and $q\left(
\ind_{\sfP'}^{\sfG'} (\rho')
\right)
=
1
=
q
\left(
\ind_{\sfP}^{\sfG} (\rho)
\right)$.
\end{subequations}
\end{proof}

For later use, we will state a more general form of Proposition~\ref{prop:comparisonofqparameters}.

\begin{corollary}
\label{corollaryofcomparisonofqparameters}
Let $f \colon \sfG \rightarrow \sfG'$ be a group homomorphism of finite groups and $\sfP'$ be a subgroup of $\sfG'$.
We write $\sfP = f^{-1}(\sfP')$.
Suppose that $f(\sfG)$ is a normal subgroup of $\sfG'$ and the quotient group $\sfG'/f(\sfG)$ is abelian.
We also suppose that $\sfG' = \sfP' \cdot f(\sfG)$.
Let $\rho'$ be an irreducible representation of $\sfP'$ and $\rho$ be an irreducible representation of $\sfP$ such that $\rho$ is contained in $\rho' \circ f \restriction_{\sfP}$.
We also suppose that the representations $\ind_{\sfP'}^{\sfG'} (\rho')$ and $\ind_{\sfP}^{\sfG} (\rho)$ have length at most two.
Then we have $
q\left(
\ind_{\sfP'}^{\sfG'} (\rho')
\right)
=
q
\left(
\ind_{\sfP}^{\sfG} (\rho)
\right)
$.
\end{corollary}
\begin{proof}
Since $\rho$ is contained in $\rho' \circ f \restriction_{\sfP}$, it is trivial on the kernel $\ker(f)$ of $f$.
We regard $\rho$ as a representation of $\sfP/\ker(f)$.
Then Proposition~\ref{prop:comparisonofqparameters} implies that we have
\[
q\left(
\ind_{\sfP'}^{\sfG'} (\rho')
\right)
=
q
\left(
\ind_{\sfP/\ker(f)}^{\sfG/\ker(f)} (\rho)
\right)
=
q
\left(
\ind_{\sfP}^{\sfG} (\rho)
\right).
\qedhere
\]
\end{proof}

\numberwithin{equation}{subsection}
\section{Calculation of $q$-parameters over finite fields}

\label{section:finitegroupcase}

In this section, we will prove a reduction result of the calculation of $q$-parameters.
We will prove that the $q$-parameter attached to the parabolic induction of an irreducible, cuspidal representation of a maximal Levi subgroup of a finite reductive group agrees with the $q$-parameter attached to the parabolic induction of a certain unipotent, cuspidal representation (see Proposition~\ref{prop:reductiontounipotentqparameters} and Corollary~\ref{corollary:reductiontounipdualver}).

\subsection{The set-up}

Let $\ff$ be a finite field of characteristic $p$.
We fix an algebraic closure $\overline{\ff}$ of $\ff$.
When we refer to an algebraic field extension $\ff'/\ff$, we always assume that $\ff' \subseteq \overline{\ff}$.
We regard the Deligne--Lusztig representations as representations over $\Coeff$ via the fixed isomorphism $\Coeff \simeq \overline{\bQ_{\ell}}$.

For a connected reductive group $\sfG$ defined over $\ff$, let $Z(\sfG)$ denote the center of $\sfG$ and write $\sfG_{\ad} = \sfG/Z(\sfG)$.
We also denote by $\sfG^{*}$ the dual group of $\sfG$ defined in \cite[Definition~5.21]{MR393266}.
For a maximal torus $\sfT$ of $\sfG$ and a character $\theta$ of $\sfT(\ff)$, let $R^{\sfG}_{\sfT}(\theta)$ denote the virtual representation of $\sfG(\ff)$ arising from the Deligne--Lusztig construction applied to $(\sfT, \theta)$.
For a semisimple element $s \in \sfG^{*}(\ff)$, let $\mathscr{E}(\sfG, s)$ denote the Lusztig series corresponding to $s$.
We write $\mathscr{E}(\sfG, 1) = \Uch(\sfG)$ and call an element of $\Uch(\sfG)$ a \emph{unipotent representation} of $\sfG(\ff)$.

Let $\sfG$ be a connected reductive group defined over $\ff$.
Let $\sfM$ be either $\sfG$ or a maximal Levi subgroup of $\sfG$.
For a maximal torus $\sfT$ of $\sfM$, we denote by $\Phi(\sfG, \sfT)$ (resp.\ $\Phi(\sfM, \sfT)$) the absolute root system of $\sfG$ (resp.\ $\sfM$) with respect to $\sfT$.
For $\alpha \in \Phi(\sfG, \sfT)$, we write $\alpha^{\vee}$ for the corresponding coroot.
We say that a character $\theta$ of $\sfT(\ff)$ is orthogonal to a coroot $\alpha^{\vee}$ if the composition $\theta \circ N_{\ff'/\ff} \circ \alpha^{\vee} \colon \ff'^{\times} \rightarrow \Coeff^{\times}$ is trivial, where $\ff'$ is a finite extension of $\ff$ such that $\sfT$ splits over $\ff'$, and $N_{\ff'/\ff}$ denotes the norm map $\sfT(\ff') \rightarrow \sfT(\ff)$.
We note that the definition of the character $\theta$ being orthogonal to $\alpha^{\vee}$ here is different from the original definition in \cite[\S5.9]{MR393266}.
We use the equivalent definition given in \cite[Lemma~3.4.14]{2019arXiv191203274K}.


Let $\rho_{\sfM}$ be an irreducible, cuspidal representation of $\sfM(\ff)$.
For a parabolic subgroup $\sfP$ of $\sfG$ with Levi factor $\sfM$, we consider the parabolically induced representation $\ind_{\sfP(\ff)}^{\sfG(\ff)} (\rho_{\sfM})$ of $\rho_{\sfM}$. (Here, we identify $\rho_{\sfM}$ with the representation of $\sfP(\ff)$ that is trivial on its unipotent radical.)
According to \cite[\S4 Corollary]{MR457579}, the isomorphism class of $\ind_{\sfP(\ff)}^{\sfG(\ff)} (\rho_{\sfM})$ only depends on $\sfG$, $\sfM$, and $\rho_{\sfM}$, but does not depend on $\sfP$.
Hence, we write this representation as $R_{\sfM}^{\sfG}(\rho_{\sfM})$.
Since $\sfM$ is either $\sfG$ or a maximal Levi subgroup of $\sfG$ and the representation $\rho_{\sfM}$ is cuspidal, we obtain from the Mackey formula that the length of the representation $R_{\sfM}^{\sfG} (\rho_{\sfM})$ is at most two.
In this section, we will study the $q$-parameter $q\left(
R_{\sfM}^{\sfG} (\rho_{\sfM})
\right)$ attached to $R_{\sfM}^{\sfG} (\rho_{\sfM})$.
We note that according to \cite[Theorem~3.18 (ii)]{MR570873} and \cite[Proposition~3.9.1]{2024arXiv240807801A}, the parameter $q\left(
R_{\sfM}^{\sfG} (\rho_{\sfM})
\right)$ appears as the parameters of Hecke algebras for finite and $p$-adic groups (see Proposition~\ref{calculationofqs} below).
We will prove in Proposition~\ref{prop:reductiontounipotentqparameters} and Corollary~\ref{corollary:reductiontounipdualver} that the calculation of the $q$-parameter $q\left(
R_{\sfM}^{\sfG} (\rho_{\sfM})
\right)$ can be reduced to the similar calculation for a unipotent representation.

\subsection{Main result of this section}
\label{subsec:statementofmainresult}

In this subsection, we will state the main result of this section.

Let $s$ be a semisimple element of $\sfM^{*}(\ff)$ such that $\rho_{\sfM} \in \mathscr{E}(\sfM, s)$.
Let $\sfG_{s}^{*}$ be the connected centralizer of $s$ in $\sfG^{*}$ and $\sfG_{s}$ be the dual of $\sfG_{s}^{*}$.
Suppose that $\sfT$ is a maximal torus of $\sfG$ and $\theta$ be a character of $\sfT(\ff)$ such that  the $\sfG(\ff)$-conjugacy class of the pair $(\sfT, \theta)$ corresponds to the the $\sfG^{*}(\ff)$-conjugacy class of the pair $(\sfT^{*}, s)$ as in \cite[(5.21.5)]{MR393266}.
Then the group $\sfG_{s}$ is a connected reductive group over $\ff$ with maximal torus $\sfT$ and root system
\[
\Phi(\sfG_{s}, \sfT) = \left\{
\alpha \in \Phi(\sfG, \sfT) \mid \text{$\theta$ is orthogonal to $\alpha^{\vee}$}
\right\}.
\]
For such a character $\theta$, we also write $\sfG_{s} = \sfG_{\theta}$ and $\sfG^{*}_{s} = \sfG^{*}_{\theta}$.

Replacing $\sfG$ with $\sfM$ in the definitions of $\sfG_{s}$ and $\sfG^{*}_{s}$, we define a Levi subgroup $\sfM_{s}$ of $\sfG_{s}$ and a Levi subgroup $\sfM^{*}_{s}$ of $\sfG^{*}_{s}$.

According to \cite[Proposition~5.1]{MR1021495}, there exists a surjective map 
\[
J^{\sfM}_{s} \colon \mathscr{E}(\sfM, s) \twoheadrightarrow \Uch(\sfM^{*}_{s}) / \sim
\]
called a \emph{Jordan decomposition}.
Here, the equivalence relations $\sim$ is defined by the conjugation of the centralizer $Z_{\sfM^{*}}(s)(\ff)$ of $s$ in $\sfM^{*}(\ff)$ on $\Uch(\sfM^{*}_{s})$.
Let $u^{*}_{\sfM}$ be a unipotent representation of $\sfM^{*}_{s}(\ff)$ contained in $J^{\sfM}_{s} \left(
\rho_{\sfM}
\right)$.
According to \cite[Theorem~3.2.22]{MR4211779}, the representation $u^{*}_{\sfM}$ is cuspidal.

We describe the construction of $u^{*}_{\sfM}$ more explicitly.
We fix a regular embedding $i \colon \sfG \rightarrow \sfG'$, that is, $i$ is an embedding of connected reductive groups defined over $\ff$ that induces an isomorphism of the adjoint groups, and the center of the group $\sfG'$ is connected.
The embedding $i$ also induces a regular embedding $i\restriction_{\sfM} \colon \sfM \rightarrow \sfM'$ for a Levi subgroup $\sfM'$ of $\sfG'$.
We regard $\sfG$ (resp.\ $\sfM$) as a subgroup of $\sfG'$ (resp.\ $\sfM'$).
We fix a semisimple element $s' \in (\sfM')^{*}(\ff)$ that maps to $s$ under the dual homomorphism $i^{*} \colon (\sfG')^{*} \to \sfG^{*}$ of $i$.
We define $(\sfM')^{*}_{s'}$ by replacing $(\sfM, s)$ with $(\sfM', s')$ in the definition of $\sfM^{*}_{s}$.
Then $i^{*}$ restricts to the surjective homomorphism $i^{*} \restriction_{(\sfM')^{*}_{s'}} \colon (\sfM')^{*}_{s'}\rightarrow \sfM_{s}^{*}$.
According to \cite[Proposition~2.3.15]{MR4211779}, we have a bijection
\begin{equation}
\label{bijectionofunipchar}
\Uch\left(
\sfM_{s}^{*}
\right) \isoarrow \Uch\left(
(\sfM')^{*}_{s'}
\right), \, u^{*} \mapsto u^{*} \circ i^{*} \restriction_{(\sfM')^{*}_{s'}(\ff)}.
\end{equation}
Let $\rho_{\sfM'}$ be an irreducible representation in $\mathscr{E}(\sfM', s')$ such that $\rho_{\sfM'} \restriction_{\sfM(\ff)}$ contains $\rho_{\sfM}$ and
let $u^{*}_{\sfM'}$ be the unipotent representation of $(\sfM')^{*}_{s'}(\ff)$ corresponding to $\rho_{\sfM'}$ via the bijection
\[
J_{s'}^{\sfM'} \colon \mathscr{E}(\sfM', s') \isoarrow \Uch((\sfM')^{*}_{s'})
\]
given in \cite[Theorem~7.1]{MR1051245}, which is a refinement of \cite[4.23 Main Theorem]{MR742472}.
Now we define $u^{*}_{\sfM}$ to be the unipotent representation of $\sfM_{s}^{*}(\ff)$ that corresponds to $u^{*}_{\sfM'}$ via \eqref{bijectionofunipchar}.

\begin{remark}
When the center of $\sfG$ is connected, the correspondence $\rho_{\sfM} \mapsto u^{*}_{\sfM}$ is uniquely determined by the conditions in \cite[Theorem~7.1]{MR1051245}.
However, if the center of $\sfG$ is disconnected, the correspondence $\rho_{\sfM} \mapsto u^{*}_{\sfM}$ above is not canonical in general.
We have to choose a Jordan decomposition $J^{\sfM}_{s}$, which \emph{a priori} depends on the choice of regular embedding $i$.
We also have to choose an irreducible representation $\rho_{\sfM'}$ of $\sfM'(\ff)$ whose restriction to $\sfM(\ff)$ contains $\rho_{\sfM}$, which corresponds to choosing a representation in the equivalence class $J^{\sfM}_{s}\left(
\rho_{\sfM}
\right)$ (see \cite[Section~9]{MR1021495}).
Recently, based on \cite{MR3416310, MR3701897}, Lusztig and Yun \cite{MR4108915} constructed a categorical version of Jordan decompositions using the theory of categorical centers of the Hecke categories.
Furthermore, Eteve \cite{2025arXiv250104113E} also gave a construction of a categorical version of Jordan decompositions using the categorical trace formalism developed in \cite{2024arXiv241213323E}.
In both constructions, the representation $\rho_{\sfM}$ is associated with a unipotent representation of an endoscopic group $H$, which is in general disconnected and has 
$\sfG_{s}$ as its identity component.
Although the categorical versions of Jordan decompositions might be more conceptual, we use the original Jordan decompositions \cite[Proposition~5.1]{MR1021495} in this paper since certain properties have so far been established only for them.
\end{remark}

The main result of this section is the following.

\begin{proposition}
\label{prop:reductiontounipotentqparameters}
We have 
\[
q\left(
R_{\sfM}^{\sfG} (\rho_{\sfM})
\right) = q\left(
R_{\sfM^{*}_{s}}^{\sfG^{*}_{s}} (u^{*}_{\sfM})
\right).
\]
\end{proposition}

We note that a similar statement is proved in \cite[Section~6]{2024arXiv241119846S} in the case where the representation $\rho_{\sfM}$ is non-singular.
We will prove Proposition~\ref{prop:reductiontounipotentqparameters} in the following two subsections.

\subsection{Reduction to simple with connected center case}
\label{subsec:reductiontosimpleconnectedcenter}

In this subsection, we will reduce Proposition~\ref{prop:reductiontounipotentqparameters} to the case that $Z(\sfG)$ is connected and $\sfG_{\ad}$ is (absolutely) simple.
\begin{remark}
We need this reduction argument to apply \cite[Corollary~4.7.6]{MR4211779}, which states that the Jordan decompositions commute with the parabolic inductions under the assumption that $Z(\sfG)$ is connected and $\sfG_{\ad}$ is simple.
In \cite{2025arXiv250104113E}, Eteve shows that his categorical version of Jordan decompositions commutes with the parabolic inductions (more generally, Deligne--Lusztig inductions), although a minor adjustment in the formalism is required. According to a private communication with Eteve, this issue will be addressed in a forthcoming revised version.
Thus, if we use Eteve's construction, we do not need the reduction argument to the case where $\sfG_{\ad}$ is simple.
On the other hand, we still need the reduction argument to the case where $Z(\sfG)$ is connected.
This is because we also would like to apply \cite[Corollary~2.6.6]{MR4211779}, where we need this assumption.
Indeed, the proof of \cite[Corollary~2.6.6]{MR4211779} uses the Deligne--Lusztig theory for the centralizer $Z_{\sfG^{*}}(s)$ of $s$ in $\sfG^{*}$, which can be disconnected when $Z(\sfG)$ is disconnected.
It might be possible to prove Proposition~\ref{prop:reductiontounipotentqparameters} directly by developing the Deligne--Lusztig theory for disconnected algebraic groups.
\end{remark}

Let $i \colon \sfG \rightarrow \sfG'$ be the regular embedding and $\rho_{\sfM'}$ be the irreducible representation of $\sfM'(\ff)$ used to define $u^{*}_{\sfM}$.
We fix a parabolic subgroup $\sfP$ of $\sfG$ with Levi factor $\sfM$ and let $\sfP'$ be the corresponding parabolic subgroup of $\sfG'$.
Then applying Proposition~\ref{prop:comparisonofqparameters} to 
$\sfG' = \sfG'(\ff)$, $\sfG = \sfG(\ff)$, $\sfP' = \sfP'(\ff)$, $\rho'= \rho_{\sfM'}$, and $\rho = \rho_{\sfM}$, we have 
\[
q\left(
R_{\sfM}^{\sfG}(\rho_{\sfM})
\right)
=
q\left(
\ind_{\sfP(\ff)}^{\sfG(\ff)} (\rho_{\sfM})
\right)
=
q\left(
\ind_{\sfP'(\ff)}^{\sfG'(\ff)} (\rho_{\sfM'})
\right)
= q\left(
R_{\sfM'}^{\sfG'} (\rho_{\sfM'})
\right).
\]
We will also reduce the calculation of the right hand side of Proposition~\ref{prop:reductiontounipotentqparameters} to the similar calculation for $(\sfG')_{s'}^{*}$.
We fix a parabolic subgroup $\sfP^{*}_{s}$ of $\sfG^{*}_{s}$ with Levi factor $\sfM^{*}_{s}$ and let $(\sfP')^{*}_{s'}$ denote the inverse image of $\sfP^{*}_{s}$ via $i^{*} \colon (\sfG')^{*} \rightarrow (\sfG)^{*}$.
Then by applying Corollary~\ref{corollaryofcomparisonofqparameters} to $f = i^{*} \restriction_{(\sfG')^{*}_{s'}(\ff)} \colon (\sfG')^{*}_{s'}(\ff) \rightarrow \sfG_{s}^{*}(\ff)$, $\sfP' = \sfP^{*}_{s}(\ff)$, $\rho' = u^{*}_{\sfM}$, and $\rho = u^{*}_{\sfM'}$, we have
\[
q\left(
R_{\sfM^{*}_{s}}^{\sfG^{*}_{s}} (u^{*}_{\sfM})
\right)
=
q\left(
\ind_{\sfP^{*}_{s}(\ff)}^{\sfG^{*}_{s}(\ff)} (u^{*}_{\sfM})
\right) = q\left(
\ind_{(\sfP')^{*}_{s'}(\ff)}^{(\sfG')^{*}_{s'}(\ff)} (u^{*}_{\sfM'})
\right)
=
q\left(
R_{(\sfM')^{*}_{s'}}^{(\sfG')^{*}_{s'}} (u^{*}_{\sfM'})
\right).
\]
Thus, Proposition~\ref{prop:reductiontounipotentqparameters} for $\sfG$ is reduced to the same claim for $\sfG'$.

Now, we suppose that $\sfG$ is a connected reductive group with connected center.
Then there exists reductive subgroups $G_{i}$ with connected centers such that the adjoint groups $G_{i, \ad}$ are $\ff$-simple, and the product map
$
p \colon \sfG_{1} \times \sfG_{2} \times \cdots \sfG_{r} \rightarrow \sfG
$
defines a surjective homomorphism with kernel a central torus.
We write $\widetilde{\sfG} = \sfG_{1} \times \sfG_{2} \times \cdots \sfG_{r}$, $\widetilde{\sfP} = p^{-1}(\sfP)$, and $\widetilde{\sfM} = p^{-1}(\sfM)$.
Since $\sfM$ is either $\sfG$ or a maximal Levi subgroup of $\sfG$, replacing the indices if necessary, we may assume that $\widetilde{\sfP} = \sfP_{1} \times \sfG_{2} \times \cdots \sfG_{r}$ and $\widetilde{\sfM} = \sfM_{1} \times \sfG_{2} \times \cdots \sfG_{r}$, where $\sfP_{1}$ is a parabolic subgroup of $\sfG_{1}$ with Levi factor $\sfM_{1}$.
Since the kernel of $p$ is a torus, Lang's theorem implies that the map $p$ restricts to a surjection $p \restriction_{\widetilde{\sfM}(\ff)} \colon \widetilde{\sfM}(\ff) \twoheadrightarrow \sfM(\ff)$.
Hence, the representation $\rho_{\sfM} \circ p \restriction_{\widetilde{\sfM}(\ff)}$ of $\widetilde{\sfM}(\ff)$ is irreducible.
We write $\rho_{\sfM} \circ p \restriction_{\widetilde{\sfM}(\ff)}$ as $\rho_{\sfM_{1}} \boxtimes \rho_{\sfG_{2}} \boxtimes \cdots \rho_{\sfG_{r}}$, where $\rho_{\sfM_{1}}$ is an irreducible, cuspidal representation of $\sfM_{1}(\ff)$, and $\rho_{\sfG_{i}}$ is an irreducible, cuspidal representation of $\sfG_{i}(\ff)$ for each $2 \le i \le r$.
Applying Corollary~\ref{corollaryofcomparisonofqparameters} to $f = p \restriction_{\widetilde{\sfG}(\ff)} \colon \widetilde{\sfG}(\ff) \rightarrow \sfG(\ff)$, $\sfP = \sfP(\ff)$, $\rho' = \rho_{\sfM}$, and $\rho = \rho_{\sfM} \circ p \restriction_{\widetilde{\sfM}(\ff)}$, we have
\begin{align}
\label{reductiontoFsimple}
q\left(
R_{\sfM}^{\sfG}(\rho_{\sfM})
\right)
&=
q\left(
\ind_{\sfP(\ff)}^{\sfG(\ff)} (\rho_{\sfM})
\right)
=
q\left(
\ind_{\widetilde{\sfP}(\ff)}^{\widetilde{\sfG}(\ff)} \left(\rho_{\sfM} \circ p \restriction_{\widetilde{\sfM}(\ff)}\right)
\right) \notag \\
&=
q\left(
\ind_{\sfP_{1}(\ff)}^{\sfG_{1}(\ff)} (\rho_{\sfM_{1}})
\right)
=
q\left(
R_{\sfM_{1}}^{\sfG_{1}}(\rho_{\sfM_{1}})
\right).
\end{align}
We will prove a similar equation in the dual side.
Let $p^{*} \colon \sfG^{*} \rightarrow \sfG^{*}_{1} \times \sfG^{*}_{2} \times \cdots \times \sfG^{*}_{r}$ denote the dual of $p$.
We write $p^{*}(s) = \widetilde{s}  = (s_{1}, s_{2}, \ldots, s_{r})$. 
According to \cite[Theorem~7.1 (vi), (vii)]{MR4211779}, we have 
\[
u^{*}_{\sfM} = \left(
u^{*}_{\sfM_{1}} \boxtimes u^{*}_{\sfG_{2}} \boxtimes \cdots \boxtimes u^{*}_{\sfG_{r}}
\right) \circ p^{*} \restriction_{\sfM^{*}_{s}(\ff)},
\]
 where the representation $u^{*}_{\sfM_{1}}$ denotes the unipotent representation of $(\sfM_{1})_{s_{1}}^{*}(\ff)$ corresponding to $\rho_{\sfM_{1}} \in \mathscr{E}(\sfM_{1}, s_{1})$, and the representation $u^{*}_{\sfG_{i}}$ denotes the unipotent representation of $(\sfG_{i})_{s_{i}}^{*}(\ff)$ corresponding to $\rho_{\sfG_{i}} \in \mathscr{E}(\sfG_{i}, s_{i})$ via \cite[Theorem~7.1]{MR1051245}.
 Let $\widetilde{\sfP}_{\widetilde{s}}^{*}$ be the parabolic subgroup of $\widetilde{\sfG}_{\widetilde{s}}^{*}$ such that $(p^{*})^{-1}\left(
 \widetilde{\sfP}_{\widetilde{s}}^{*}
 \right) = \sfP_{s}^{*}$.
 Then we have $\widetilde{\sfP}_{\widetilde{s}}^{*} = (\sfP_{1})_{s_{1}}^{*} \times (\sfG_{2})_{s_{2}}^{*} \times \cdots \times (\sfG_{r})_{s_{r}}^{*}$ for a parabolic subgroup $(\sfP_{1})_{s_{1}}^{*} $ of $(\sfG_{1})_{s_{1}}^{*}$ with Levi factor $(\sfM_{1})_{s_{1}}^{*}$. 
 Applying Corollary~\ref{corollaryofcomparisonofqparameters} to $f = p^{*} \restriction_{\sfG_{s}^{*}(\ff)} \colon \sfG_{s}^{*}(\ff) \rightarrow \widetilde{\sfG}_{\widetilde{s}}^{*}(\ff)$, $\sfP =  \widetilde{\sfP}_{\widetilde{s}}^{*}(\ff)$, $\rho' = u^{*}_{\sfM_{1}} \boxtimes u^{*}_{\sfG_{2}} \boxtimes \cdots \boxtimes u^{*}_{\sfG_{r}}$, and $\rho = u^{*}_{\sfM}$, we have
 \begin{align*}
 q\left(
R_{\sfM^{*}_{s}}^{\sfG^{*}_{s}} (u^{*}_{\sfM})
\right)
&=
 q\left(
\ind_{\sfP^{*}_{s}(\ff)}^{\sfG^{*}_{s}(\ff)} (u^{*}_{\sfM})
\right)
=
q\left(
\ind_{\widetilde{\sfP}^{*}_{\widetilde{s}}(\ff)}^{\widetilde{\sfG}^{*}_{\widetilde{s}}(\ff)} (u^{*}_{\sfM_{1}} \boxtimes u^{*}_{\sfG_{2}} \boxtimes \cdots \boxtimes u^{*}_{\sfG_{r}})
\right) \\
&=
q\left(
\ind_{(\sfP_{1})_{s_{1}}^{*}(\ff)}^{(\sfG_{1})_{s_{1}}^{*}(\ff)} (u^{*}_{\sfM_{1}})
\right)
=
q\left(
R_{(\sfM_{1})_{s_{1}}^{*}}^{(\sfG_{1})_{s_{1}}^{*}} (u^{*}_{\sfM_{1}})
\right).
 \end{align*}
Combining this with \eqref{reductiontoFsimple}, Proposition~\ref{prop:reductiontounipotentqparameters} for $\sfG$ is reduced to the same claim for $\sfG_{1}$.

Now we assume that $Z(\sfG)$ is connected and $\sfG_{\ad}$ is $\ff$-simple.
In this case, we have a surjective homomorphism $\Res_{\ff'/\ff} (\sfH) \to \sfG$ with kernel a central torus, where $\sfH$ is an algebraic group defined over a finite extension $\ff'$ of $\ff$ such that $Z(\sfH)$ is connected and $\sfH_{\ad}$ is absolutely simple.
Using the same argument as above, we may suppose that $\sfG$ is of the form $\Res_{\ff'/\ff} (\sfH)$.
Since $\sfM$ and $\sfT$ are defined over $\ff$, we have $\sfM = \Res_{\ff'/\ff} (\sfM_{\sfH})$ and $\sfT = \Res_{\ff'/\ff} (\sfT_{\sfH})$ for a Levi subgroup $\sfM_{\sfH}$ of $\sfH$ and a maximal torus $\sfT_{\sfH}$ of $\sfM_{\sfH}$.
We regard $s$ with an element of $\sfT_{\sfH}^{*}(\ff') = \Res_{\ff'/\ff}(\sfT_{\sfH}^{*})(\ff) = \sfT^{*}(\ff)$.
Then we have, $\sfG^{*}_{s} = \Res_{\ff'/\ff} \left(
\sfH^{*}_{s}
\right)$ and $\sfM^{*}_{s} = \Res_{\ff'/\ff} \left(
(\sfM_{\sfH})^{*}_{s}
\right)$.
According to \cite[Corollary~8.8]{MR3935811}, we have $\mathscr{E}(\sfM, s) = \mathscr{E}(\sfM_{\sfH}, s)$ as sets of irreducible representations of $\sfM(\ff) = \sfM_{\sfH}(\ff')$, and $\Uch(\sfM^{*}_{s}) = \Uch((\sfM_{\sfH})_{s}^{*})$ as sets of irreducible representations of $\sfM^{*}_{s}(\ff) = (\sfM_{\sfH})_{s}^{*}(\ff')$.
Since the conditions in \cite[Theorem~7.1]{MR1051245} for $\sfM$ and $\sfM_{\sfH}$ are equivalent, we obtain that the bijections
\[
J_{s}^{\sfM} \colon \mathscr{E}(\sfM, s) \isoarrow \Uch(\sfM_{s}^{*})
\qquad
\text{and}
\qquad
J_{s}^{\sfM_{\sfH}} \colon \mathscr{E}(\sfM_{\sfH}, s) \isoarrow \Uch((\sfM_{\sfH})_{s}^{*})
\]
agree.
Thus, we have reduced Proposition~\ref{prop:reductiontounipotentqparameters} for $\sfG$ to the same claim for $\sfH$.

\subsection{Proof of Proposition~\ref{prop:reductiontounipotentqparameters}}
\label{subsec:connectedcenter}

In this subsection, we will prove Proposition~\ref{prop:reductiontounipotentqparameters}.
By the reduction arguments in Section~\ref{subsec:reductiontosimpleconnectedcenter}, we may suppose that $Z(\sfG)$ is connected and $\sfG_{\ad}$ is simple.
In this case, according to \cite[Corollary~4.7.6]{MR4211779}, the Jordan decompositions $J^{\sfG}_{s}$ and $J^{\sfM}_{s}$ commute with the parabolic inductions.
In particular, $J^{\sfG}_{s}$
restricts to a bijection
\[
\left\{
\rho \in \mathscr{E}(\sfG, s) \mid \Hom_{\sfG(\ff)}\left(
\rho, R_{\sfM}^{\sfG} (\rho_{\sfM})
\right) \neq \{0\}
\right\} \isoarrow
\left\{
u^{*} \in \Uch(\sfG^{*}_{s}) \mid \Hom_{\sfG^{*}_{s}(\ff)}\left(
u^{*}, R_{\sfM^{*}_{s}}^{\sfG^{*}_{s}} (u^{*}_{\sfM})
\right) \neq \{0\}
\right\}.
\]
Hence, the lengths of the representations $R_{\sfM}^{\sfG} (\rho_{\sfM})$ and $R_{\sfM^{*}_{s}}^{\sfG^{*}_{s}} (u^{*}_{\sfM})$, which are at most two, are the same.
Suppose that they have length two, and we write $R_{\sfM}^{\sfG} (\rho_{\sfM}) = \rho_{1} \oplus \rho_{2}$ with $\dim(\rho_{1}) \ge \dim(\rho_{2})$ and $R_{\sfM^{*}_{s}}^{\sfG^{*}_{s}} (u^{*}_{\sfM}) = u^{*}_{1} \oplus u^{*}_{2}$ with $\dim(u^{*}_{1}) \ge \dim(u^{*}_{2})$.
Then according to \cite[Corollary~2.6.6]{MR4211779}, we have $\frac{\dim(\rho_{1})}{\dim(\rho_{2})} = \frac{\dim(u^{*}_{1})}{\dim(u^{*}_{2})}$.
Thus, we obtain the claim.

\subsection{A generalization of Proposition~\ref{prop:reductiontounipotentqparameters}}

In this subsection, we will generalize Proposition~\ref{prop:reductiontounipotentqparameters} to Corollary~\ref{corollary:reductiontounipdualver} for later use.
The reason why such a generalization is necessary will be explained in Remark~\ref{remarkwhyuptotakingduals} below.

\begin{definition}
\label{def:almostisomforgroups}
Let $\sfG_{1}$ and $\sfG_{2}$ be connected reductive groups defined over $\ff$.
We decompose the adjoint groups of $\sfG_{1}$ and $\sfG_{2}$ into $\ff$-simple factors as
\[
\sfG_{1, \ad} = \prod_{1 \le i \le m} \sfG_{1, i}
\qquad
\text{and}
\qquad
\sfG_{2, \ad} = \prod_{1 \le j \le n} \sfG_{2, j},
\]
where $\sfG_{1, i}$ and $\sfG_{2, j}$ are $\ff$-simple subgroups of $\sfG_{1, \ad}$ and $\sfG_{2, \ad}$, respectively.
\begin{enumerate}[(1)]
\item
We say that the groups $\sfG_{1}$ and $\sfG_{2}$ are \emph{adjointly isomorphic up to taking duals} and write $\sfG_{1} \sim_{\dual} \sfG_{2}$ if $m = n$, and for each $1 \le i \le n$, one of the conditions $\sfG_{1, i} \simeq \sfG_{2, i}$ or $(\sfG^{*}_{1, i})_{\ad} \simeq \sfG_{2, i}$ holds.
\item
Let $\sfM_{1}$ (resp.\ $\sfM_{2}$) be a Levi subgroup of $\sfG_{1}$ (resp.\ $\sfG_{2}$).
We write $\sfM_{1, \sfG-\ad}$ (resp.\ $\sfM_{2, \sfG-\ad}$) for the image of $\sfM_{1}$ (resp.\ $\sfM_{2}$) in $\sfG_{1, \ad}$ (resp.\ $\sfG_{2, \ad}$), and let $\sfM_{1, i}$ (resp.\ $\sfM_{2, i}$) denote the projection of $\sfM_{1, \sfG-\ad}$ (resp.\ $\sfM_{2, \sfG-\ad}$) to $\sfG_{1, i}$ (resp.\ $\sfG_{2, i}$).
We say that the pairs $(\sfG_{1}, \sfM_{1})$ and $(\sfG_{2}, \sfM_{2})$ are \emph{adjointly isomorphic up to taking duals} and write $(\sfG_{1}, \sfM_{1}) \sim_{\dual} (\sfG_{2}, \sfM_{2})$ if $\sfG_{1} \sim_{\dual} \sfG_{2}$, and for each $i$, one of the following conditions holds:
\begin{itemize}
\item
we have an isomorphism $\sfG_{1, i} \simeq \sfG_{2, i}$ that restricts to an isomorphism $\sfM_{1, i} \simeq \sfM_{2, i}$;
\item
we have an isomorphism $(\sfG^{*}_{1, i})_{\ad} \simeq \sfG_{2, i}$ that restricts to an isomorphism $(\sfM^{*}_{1, i})_{\sfG-\ad} \simeq \sfM_{2, i}$, where $(\sfM^{*}_{1, i})_{\sfG-\ad}$ denotes the image of $\sfM^{*}_{1, i}$ in $(\sfG^{*}_{1, i})_{\ad}$.
\end{itemize}
\end{enumerate}
\end{definition}

When we write $\sfG_{1} \sim_{\dual} \sfG_{2}$ or $(\sfG_{1}, \sfM_{1}) \sim_{\dual} (\sfG_{2}, \sfM_{2})$ below, we always fix isomorphisms $\sfG_{1, i} \simeq \sfG_{2, i}$ and $(\sfG^{*}_{1, i})_{\ad} \simeq \sfG_{2, i}$ above.

We record a criterion for the pairs $(\sfG_{1}, \sfM_{1})$ and $(\sfG_{2}, \sfM_{2})$ to be adjointly isomorphic up to taking duals in terms of root systems.
\begin{lemma}
\label{lemmacriterionforalmostisomorphic}
Let $\sfG_1$ and $\sfG_2$ be connected reductive groups over $\ff$.
Let $\sfT_{1}$ (resp.\ $\sfT_{2}$) be a maximal torus of $\sfG_{1}$ (resp.\ $\sfG_{2}$).
Suppose that the root systems $\Phi(\sfG_1, \sfT_1)$ and $\Phi(\sfG_2, \sfT_2)$ decompose into subsystems
\[
\Phi(\sfG_1, \sfT_1) = \bigsqcup_{1 \le i \le n} \Phi_{1, i} 
\qquad
\text{and}
\qquad
\Phi(\sfG_2, \sfT_2) = \bigsqcup_{1 \le i \le n} \Phi_{2, i} 
\]
such that  for each $i$ we have:
\begin{enumerate}[(1)]
\item the subsystem $\Phi_{1, i}$ (resp.\ $\Phi_{2, i}$) is a single $\Gal(\overline{\ff}/\ff)$-orbit of an irreducible component of $\Phi(\sfG_1, \sfT_1)$ (resp.\ $\Phi(\sfG_2, \sfT_2)$);
\item one of the conditions $\Phi_{1, i} \simeq \Phi_{2, i}$ or $\Phi_{1, i}^{\vee} \simeq \Phi_{2, i}$ holds, where $\Phi_{1, i}^{\vee}$ denotes the dual root system of $\Phi_{1, i}$, and the isomorphisms are as $\Gal(\overline{\ff}/\ff)$-sets.
\end{enumerate}
Then the groups $\sfG_1$ and $\sfG_2$ are adjointly isomorphic up to taking duals.
Moreover, suppose that $\sfM_1$ (resp.\ $\sfM_2$) is a Levi subgroup of $\sfG_1$ (resp.\ $\sfG_2$) containing $\sfT_1$ (resp.\ $\sfT_2$) such that for each $i$, at least one of the following conditions holds:
\begin{itemize}
\item
we have an isomorphism of $\Gal(\overline{\ff}/\ff)$-sets $\Phi_{1, i} \simeq \Phi_{2, i}$ that restricts to an isomorphism $\Phi_{1, i} \cap \Phi(M_{1}, T_{1}) \simeq \Phi_{2, i} \cap \Phi(M_{2}, T_{2})$;
\item
we have an isomorphism of $\Gal(\overline{\ff}/\ff)$-sets $\Phi_{1, i}^{\vee} \simeq \Phi_{2, i}$ that restricts to an isomorphism $\Phi_{1, i}^{\vee} \cap \Phi(M_{1}, T_{1})^{\vee} \simeq \Phi_{2, i} \cap \Phi(M_{2}, T_{2})$, where $\Phi(M_{1}, T_{1})^{\vee}$ denotes the dual root system of $\Phi(M_{1}, T_{1})$.
\end{itemize}
Then the pairs $(\sfG_{1}, \sfM_{1})$ and $(\sfG_{2}, \sfM_{2})$ are adjointly isomorphic up to taking duals.
\end{lemma}
\begin{proof}
According to \cite[Lemma~1.4.24]{MR4211779}, the isomorphism class of an adjoint group over $\ff$ is determined by its root system and the Galois action on it.
Thus, we obtain the claim.
\end{proof}

Suppose that $(\sfG_{1}, \sfM_{1}) \sim_{\dual} (\sfG_{2}, \sfM_{2})$.
For $k \in \{1, 2\}$, let $u_{k}$ be a unipotent representation of $\sfM_{k}(\ff)$. 
According to \cite[Proposition~2.3.15]{MR4211779}, there exists a unique unipotent representation $u_{k, \ad}$ of $\sfM_{k, \sfG-\ad}(\ff)$ such that the restriction of $u_{k, \ad}$ to the image of $\sfM_{k}(\ff)$ agrees with $u_{k}$.
We decompose the representations $u_{1, \ad}$ and $u_{2, \ad}$ as 
\[
u_{1, \ad} = u_{1, 1} \boxtimes u_{1, 2} \boxtimes \cdots \boxtimes u_{1, n}
\qquad
\text{and}
\qquad
u_{2, \ad} = u_{2, 1} \boxtimes u_{2, 2} \boxtimes \cdots \boxtimes u_{2, n},
\]
where $u_{1, i}$ (resp.\ $u_{2, i}$) denotes a unipotent representation of $\sfM_{1, i}(\ff)$ (resp.\ $\sfM_{2, i}(\ff)$).

\begin{definition}
We say that the representations $u_{1}$ and $u_{2}$ are \emph{adjointly isomorphic up to taking duals} and write $u_{1} \sim_{\dual} u_{2}$ if for each $i$, the following conditions hold:
\begin{itemize}
\item
if $\sfG_{1, i} \simeq \sfG_{2, i}$, we have $u_{1} \simeq u_{2}$;
\item
if $(\sfG^{*}_{1, i})_{\ad} \simeq \sfG_{2, i}$, we have $J^{\sfM_{1, i}}_{1}(u_{1, i})_{\sfG-\ad} \simeq u_{2}$, where $J^{\sfM_{1, i}}_{1}(u_{1, i})$ denotes the unipotent representation of $\sfM^{*}_{1, i}(\ff)$ that corresponds to $u_{1, i}$ via the Jordan decomposition $J^{\sfM_{1, i}}_{1} \colon \Uch(\sfM_{1, i}) \isoarrow \Uch(\sfM^{*}_{1, i})$, and $J^{\sfM_{1, i}}_{1}(u_{1, i})_{\sfG-\ad}$ denotes the unique unipotent representation of $(\sfM^{*}_{1, i})_{\sfG-\ad}(\ff)$ whose restriction to the image of $\sfM^{*}_{1, i}(\ff)$ is $J^{\sfM_{1, i}}_{1}(u_{1, i})$.
\end{itemize} 
\end{definition}

\begin{remark}
Suppose that $(\sfG_{1}, \sfM_{1}) \sim_{\dual} (\sfG_{2}, \sfM_{2})$ and let $u_{1}$ be a unipotent representation of $\sfM_{1}(\ff)$.
Then there exists a unique unipotent representation $u_{2}$ of $\sfM_{2}(\ff)$ such that $u_{1} \sim_{\dual} u_{2}$.
Moreover, if $u_{1}$ is cuspidal, so is $u_{2}$ by \cite[Theorem~3.2.22]{MR4211779}.
\end{remark}

\begin{lemma}
\label{lemma:qparameterstakingduals}
Let $\sfG_{1}$ and $\sfG_{2}$ be a connected reductive group over $\ff$ and $\sfM_{1}$ (resp.\ $\sfM_{2}$) be a Levi subgroup of $\sfG_{1}$ (resp.\ $\sfG_{2}$) such that $(\sfG_{1}, \sfM_{1}) \sim_{\dual} (\sfG_{2}, \sfM_{2})$.
We suppose that $\sfM_{1}$ is either $\sfG_{1}$ or a maximal Levi subgroup of $\sfG_{1}$.
Let $u_{1}$ (resp.\ $u_{2}$) be a unipotent, cuspidal representation of $\sfM_{1}(\ff)$ (resp.\ $\sfM_{2}(\ff)$) such that $u_{1} \sim_{\dual} u_{2}$.
Then we have 
\[
q\left(
R_{\sfM_{1}}^{\sfG_{1}} (u_{1})
\right) = q\left(
R_{\sfM_{2}}^{\sfG_{2}} (u_{2})
\right).
\]
\end{lemma}
\begin{proof}
By using Proposition~\ref{prop:comparisonofqparameters}, we can reduce the claim to the same claim for $\sfG_{1, \ad}$ and $\sfG_{2, \ad}$.
Hence, it suffices to prove the equation for each $\ff$-simple factors of $\sfG_{1, \ad}$ and $\sfG_{2, \ad}$. 
We fix $1 \le i \le n$.
If $\sfG_{1, i} \simeq \sfG_{2, i}$, the claim is trivial.
Suppose that $(\sfG^{*}_{1, i})_{\ad} \simeq \sfG_{2, i}$.
Then applying Proposition~\ref{prop:reductiontounipotentqparameters} to $(\sfM, s, \rho_{\sfM}) = (\sfM_{1, i}, 1, u_{1, i})$, we have
\[
q\left(
R_{\sfM_{1, i}}^{\sfG_{1, i}} (u_{1, i})
\right) = q\left(
R_{\sfM^{*}_{1, i}}^{\sfG^{*}_{1, i}} (J^{\sfM_{1, i}}_{1}(u_{1, i}))
\right).
\]
Applying Proposition~\ref{prop:comparisonofqparameters} again, we obtain that 
\[
q\left(
R_{\sfM^{*}_{1, i}}^{\sfG^{*}_{1, i}} (J^{\sfM_{1, i}}_{1}(u_{1, i}))
\right)
=
q\left(
R_{(\sfM^{*}_{1, i})_{\sfG-\ad}}^{(\sfG^{*}_{1, i})_{\ad} } (J^{\sfM_{1, i}}_{1}(u_{1, i})_{\sfG-\ad})
\right).
\]
Moreover, since $u_{1} \sim_{\dual} u_{2}$, we have
\[
R_{(\sfM^{*}_{1, i})_{\sfG-\ad}}^{(\sfG^{*}_{1, i})_{\ad}} (J^{\sfM_{1, i}}_{1}(u_{1, i})_{\sfG-\ad})
\simeq
R_{\sfM_{2, i}}^{\sfG_{2, i}} (u_{2, i})
\]
as representations of $(\sfG^{*}_{1, i})_{\ad}(\ff) \simeq \sfG_{2, i}(\ff)$.
Thus, we obtain the claim.
\end{proof}

Now, we generalize Proposition~\ref{prop:reductiontounipotentqparameters}.

\begin{corollary}
\label{corollary:reductiontounipdualver}
Let $\sfG'_{s}$ be a connected reductive group over $\ff$ and $\sfM'_{s}$ be a Levi subgroup of $\sfG'_{s}$ such that $(\sfG^{*}_{s}, \sfM^{*}_{s}) \sim_{\dual} (\sfG'_{s}, \sfM'_{s})$.
Let $u'_{\sfM}$ be a unipotent, cuspidal representation of $\sfM'_{s}$ such that $u^{*}_{\sfM} \sim_{\dual} u'_{\sfM}$.  
Then we have
\[
q\left(
R_{\sfM}^{\sfG} (\rho_{\sfM})
\right) = q\left(
R_{\sfM'_{s}}^{\sfG'_{s}} (u'_{\sfM})
\right).
\]
\end{corollary}

\begin{proof}
The corollary follows from Proposition~\ref{prop:reductiontounipotentqparameters} and Lemma~\ref{lemma:qparameterstakingduals}.
\end{proof}

\section{Parameters and affine Hecke algebras for $p$-adic groups}
\label{section:heckealgebraisom}

In this section, we will prove the main theorem of this paper.
Let $F$ be a non-archimedean local field and $G$ be a connected reductive group defined over $F$.
We will prove that the affine Hecke algebra appearing in the description of the Hecke algebra attached to a depth zero type for $G$ in \cite[Theorem~5.3.6]{2024arXiv240807801A} is isomorphic to the affine Hecke algebra appearing in the description of the Hecke algebra attached to a unipotent type for a reductive group $G_{\theta}$ that splits over an unramified extension of $F$.
By combining this with the Hecke algebra isomorphisms in \cite[Theorem~4.4.1]{2024arXiv240807805A}, we will also obtain that the affine Hecke algebra appearing in the description of the Hecke algebra attached to a type constructed by Kim and Yu is isomorphic to the affine Hecke algebra attached to a unipotent type.
In particular, we will prove a version of Lusztig's conjecture about parameters of Hecke algebras assuming that $G$ splits over a tamely ramified extension of $F$ and $p \nmid \abs{W}$.

\subsection{Notation}

Let $F$ be a non-archimedean local field with residue characteristic $p$.
We write $\ff$ for the residue field of $F$.
We fix a separable closure $\overline{F}$ of $F$.
When we refer to a separable field extension $E/F$, we always assume that $E \subseteq \overline{F}$.
We write $F^{\unr}$ for the maximal unramified extension of $F$.
For an algebraic extension $E$ of $F$, let $\cO_{E}$ denote the ring of integers of $E$.

For a linear algebraic group $G$ over $F$ and an algebraic field extension $E$ of $F$, we write $G_{E}$ for the base change of $G$ to $E$.

Suppose that $G$ is a connected reductive group defined over $F$.
We write $Z(G)$ for the center of $G$ and $A_{G}$ for the maximal split torus in $Z(G)$.
For a torus $T$ of $G$,
we denote by $X^{*}(T)$ and $X_{*}(T)$ the (algebraic) character group and cocharacter group of $T_{\overline F}$, respectively.
For a compact, open subgroup $K$ of $G(F)$ and an irreducible representation $\rho$ of $K$, let $\cH(G(F), (K, \rho))$ denote the Hecke algebra attached to $(K, \rho)$.
We refer to \cite[Section~2.2]{2024arXiv240807801A} for the precise definition of $\cH(G(F), (K, \rho))$.

Let $E = F$ or $F^{\unr}$.
For a maximal $E$-split $F$-torus $T$, let $\Phi(G, T, E)$ denote the relative root system of $G_{E}$ with respect to $T_{E}$.
We also let $\Phi_{\aff}(G, T, E)$ denote the relative affine root system associated to $(G_{E}, T_{E})$ by the work of Bruhat and Tits (\cite{MR327923}).
For $\alpha \in \Phi(G, T, E)$, let $\alpha^{\vee}$ be the corresponding coroot.
For $a \in \Phi_{\aff}(G, T, E)$, let $D_{T}(a) \in \Phi(G, T, E)$ denote the gradient of $a$.

We denote by $\cB(G, E)$ the enlarged Bruhat--Tits building of $G_{E}$, and for a maximal $E$-split $F$-torus $T$ of $G$, we denote by $\cA(G, T, E)$ the apartment of $T_{E}$ in $\cB(G, E)$.
We also write $\cB^{\red}(G, E)$ for the reduced building of $G_{E}$ and $\cA^{\red}(G, T, E)$ for the apartment of $T_{E}$ in $\cB^{\red}(G, E)$.
For $x \in \cB(G, E)$, we denote by $[x]_{G}$ the image of $x$ in $\cB^{\red}(G, E)$.
Let $x \in \cB(G, E)$.
For any abstract group $H$ that acts on $G(E)$, and thus on $\cB(G,E)$ and $\cB^{\red}(G,E)$,
we let $H_{x}$ and $H_{[x]_G}$ denote the stabilizers of $x$ and $[x]_G$ under the actions of $H$, respectively.
We denote by $G(E)_{x,0}$ the group of $\cO_{E}$-points of the connected parahoric group scheme
of $G$ associated to the point $x$,
and by $G(E)_{x,0+}$ the pro-$p$ radical of $G(E)_{x,0}$.
The reductive quotient of the special fiber of the connected parahoric group scheme above will be denoted by $\sfG^{\circ}_{x}$.
Thus, $\sfG^{\circ}_{x}(\overline{\ff}) = G(F^{\unr})_{x, 0}/G(F^{\unr})_{x, 0+}$, and $\sfG^{\circ}_{x}(\ff) = G(F)_{x,0} / G(F)_{x, 0+}$ if $x \in \cB(G, F)$.
In the case that $G$ is a torus, we omit the notation $x$ and write $G(E)_{x, 0} = G(E)_{0}$, $G(E)_{x, 0+} = G(E)_{0+}$, and $\sfG^{\circ}_{x} = \sfG^{\circ}$, which do not depend on $x$.

Let $\IEC(G)$ denote the set of \emph{inertial equivalence
classes},
i.e., equivalence classes $[L,\sigma]_G$ of cuspidal
pairs $(L,\sigma)$ in $G$, where $L$ is a Levi subgroup of $G$,
$\sigma$ is an irreducible, supercuspidal representation of $L(F)$,
and where the equivalence is given by conjugation by $G(F)$
and twisting by unramified characters of $L(F)$.
Let $K$ be a compact, open subgroup of $G(F)$ and $\rho$ is an irreducible representation of $K$.
We say that $(K, \rho)$ is a \emph{depth-zero type} if there exists $x \in \cB(G, F)$ and an irreducible, cuspidal representation $\rho_{x, 0}$ of $\sfG^{\circ}_{x}(\ff)$ such that $G(F)_{x, 0} \subseteq K \subseteq G(F)_{x}$ and the restriction of $\rho$ to $G(F)_{x, 0}$ contains the inflation of $\rho_{x, 0}$ to $G(F)_{x, 0}$.
A depth-zero type is an $\fS$-type for 
a finite subset $\fS$ of $\IEC(G)$ in the sense of \cite[Section~4]{BK-types}.
If the representation $\rho_{x, 0}$ is a unipotent representation of $\sfG^{\circ}_{x}(\ff)$, we say that $(K, \rho)$ is a \emph{unipotent type}.

\subsection{Review of \cite{2024arXiv240807801A}}
\label{subsec:reviewofAFMO}

In this section, we recall the description of the Hecke algebras attached to depth-zero types given in \cite[Section~5]{2024arXiv240807801A}, which is a generalization of \cite[Theorem 7.12]{Morris}.

Let $G$ be a connected reductive group defined over $F$ and $M$ be a Levi subgroup of $G$.
Let $x_{0} \in \cB(M, F)$ such that the image $[x_{0}]_{M}$ under the projection to $\cB^{\red}(M, F)$ is a vertex.
We fix a $0$-generic admissible embedding $\iota \colon \cB(M, F) \hookrightarrow \cB(G, F)$ relative to $x_{0}$ in the sense of \cite[Definition~3.2]{Kim-Yu}, that is, $M(F)_{x_0, 0}/M(F)_{x_0, 0+} \simeq G(F)_{\iota(x_0), 0}/G(F)_{\iota(x_0), 0+}$.
We use this embedding to identify $\cB(M, F)$ with its image in $\cB(G, F)$.
Recalling that $A_{M}$ denotes the maximal split torus in the center $Z(M)$ of $M$,
we write $\cA_{x_0} = x_{0} + (X_{*}(A_{M}) \otimes_{\bZ} \bR) \subseteq \cB(M, F)$ (see \cite[Section~3.1]{2024arXiv240807801A}).

Let $\rho_{M, 0}$ be an irreducible, cuspidal representation of $\sfM^{\circ}_{x_{0}}(\ff)$.
We identify $\rho_{M, 0}$ with its inflation to $M(F)_{x_0, 0}$.
We also introduce a ``disconnected version'' $\rho_{M, \dc}$ of $\rho_{M, 0}$.
Let $\rho_{M, \dc}$ be an irreducible representation of $M(F)_{x_{0}}$ such that $\rho_{M, \dc} \restriction_{M(F)_{x_{0}, 0}}$ contains $\rho_{M, 0}$.
Let $(K_{M}, \rho_{M})$ be either $(M(F)_{x_0, 0}, \rho_{M, 0})$ or $(M(F)_{x_0}, \rho_{M, \dc})$.
Then the triple $\left(
(G, M), (x_{0}, \iota), (K_{M}, \rho_{M})
\right)$ is a depth-zero $G$-datum in the sense of \cite[Section~7.1]{Kim-Yu}.

Let $\sfT^{\circ}$ be a maximal torus of $\sfM^{\circ}_{x_{0}}$ and $\theta$ be a character of $\sfT^{\circ}(\ff)$ such that $\rho_{M, 0}$ occurs in the Deligne--Lusztig induction $R^{\sfM^{\circ}_{x_{0}}}_{\sfT^{\circ}}(\theta)$.
Since $\rho_{M, 0}$ is cuspidal, the torus $\sfT^{\circ}$ is an elliptic torus of $\sfM^{\circ}_{x_0}$.
According to \cite[Lemma~2.3.1]{MR2214792}, there exists a maximal $F^{\unr}$-split torus $T$ of $M$ such that $x_{0} \in \cA(M, T, F^{\unr})$ and $\sfT^{\circ}(\ff) = T(F)_{0}/T(F)_{0+}$.
Since $\sfT^{\circ}$ is an elliptic torus of $\sfM^{\circ}_{x_0}$ and $[x_{0}]_{M}$ is a vertex of $\cB^{\red}(M, F)$, the torus $T$ is an elliptic torus of $M$.
We identify $\theta$ with its inflation to $T(F)_{0}$.

In \cite[Section~5.1]{2024arXiv240807801A}, we constructed a family $\cK$ of compact, open subgroups of $G(F)$ and their irreducible smooth representations, which we recall here briefly.
For $x \in \cA_{x_{0}}$, we set $K_{x} = K_{M} \cdot G(F)_{x, 0}$ and $K_{x, +} = G(F)_{x, 0+}$.
Let $\cA_{\gen}$ be the set of $x \in \cA_{x_{0}}$ such that the embedding $\iota$ is $0$-generic relative to $x$.
For $x \in \cA_{\gen}$, we define an irreducible smooth representation $\rho_{x}$ of $K_{x}$ as the composition of $\rho_{M}$ with the inverse of the isomorphism $K_{M}/M(F)_{x, 0+} \isoarrow K_{x}/K_{x, 0+}$ that comes from the inclusion $K_{M} \subseteq K_{x}$.
According to \cite[Section~7.1]{Kim-Yu}, which is based on \cite{MR1371680} and \cite{BK-types}, for $x \in \cA_{\gen}$, the pair $(K_{x}, \rho_{x})$ is a depth-zero type attached to a single Bernstein block if $(K_{M}, \rho_{M}) = (M(F)_{x_0}, \rho_{M, \dc})$.
In the case that $(K_{M}, \rho_{M}) = (M(F)_{x_0, 0}, \rho_{M, 0})$, the pair $(K_{x}, \rho_{x})$ is a depth-zero type attached to a finite set of Bernstein blocks.
We define 
$
\cK = 
\left\{
(K_{x}, K_{x, +}, \rho_{x})
\right\}_{x \in \cA_{\gen}}
$.

In \cite[Theorem~5.3.6]{2024arXiv240807801A}, we gave an explicit description of the Hecke algebra $\cH(G(F), (K_{x_0}, \rho_{x_0}))$ as a semi-direct product of an affine Hecke algebra with a twisted group algebra.
When $(K_{M}, \rho_{M}) = (M(F)_{x_0, 0}, \rho_{M, 0})$, such description was obtained by Morris in \cite[Theorem 7.12]{Morris}.
To explain \cite[Theorem~5.3.6]{2024arXiv240807801A}, we prepare some notation.
Following \cite[Sections~3.3, 3.4, 5.3]{2024arXiv240807801A}, we define the groups $\Nheart$ and $\Wheart$ by
\[
\Nheart = N_{G(F)}(\rho_{M}) \cap N_{G}(M)(F)_{[x_0]_{M}}
\qquad
\text{and}
\qquad
\Wheart = \Nheart / K_{M}.
\]
We note that the action of $N_{G}(M)(F)_{[x_0]_{M}}$ on $\cA_{x_0}$ induces an action of $\Wheart$ on $\cA_{x_0}$.
Next, we recall the definition of the normal subgroup $W(\rho_{M})_{\aff}$ of $\Wheart$ from \cite[Sections~3.7, 5.3]{2024arXiv240807801A}.
We fix a maximal split torus $S$ of $M$ such that $x_{0} \in \cA(M, S, F)$.
We define the set $\Phi_{\aff}(G, A_{M})$ of affine functionals on $\cA_{x_0}$ by
\[
\Phi_{\aff}(G, A_{M}) = 
\left\{
a \restriction_{\cA_{x_{0}}} \mid a \in \Phi_{\aff}(G, S, F) \smallsetminus \Phi_{\aff}(M, S, F)
\right\}.
\]
According to \cite[Lemma~5.2.1]{2024arXiv240807801A}, the set $\Phi_{\aff}(G, A_{M})$ does not depend on the choice of $S$.
For later use, we record another description of $\Phi_{\aff}(G, A_{M})$.
\begin{lemma}
\label{lemma:descriptionofphiaffGaM}
We have
\[
\Phi_{\aff}(G, A_{M}) = \left\{
a \restriction_{\cA_{x_{0}}} \mid a \in \Phi_{\aff}(G, T, F^{\unr}) \smallsetminus \Phi_{\aff}(M, T, F^{\unr})
\right\}.
\]
\end{lemma}
\begin{proof}
Let $T'$ be a maximal $F^{\unr}$-split torus of $M$ that contains a maximal $F$-split torus $S$ of $M$  such that $x_{0} \in \cA(M, S, F)$.
Such a torus exists by \cite[Proposition~9.3.4]{KalethaPrasad}.
According to \cite[Proposition~9.4.28]{KalethaPrasad}, we have
\[
\Phi_{\aff}(G, S, F) = \left\{
a \restriction_{\cA(G, S, F)} \mid a \in \Phi_{\aff}(G, T', F^{\unr}), \, \text{$a \restriction_{\cA(G, S, F)}$ is non-constant}
\right\}.
\]
Thus, we obtain that
\begin{align*}
\Phi_{\aff}(G, A_{M}) &= \left\{
a \restriction_{\cA_{x_{0}}} \mid a \in \Phi_{\aff}(G, T', F^{\unr}), \, \text{$a \restriction_{\cA_{x_{0}}}$ is non-constant}
\right\} \\
&= \left\{
a \restriction_{\cA_{x_{0}}} \mid a \in \Phi_{\aff}(G, T', F^{\unr}) \smallsetminus \Phi_{\aff}(M, T', F^{\unr})
\right\}.
\end{align*}
Since $T$ and $T'$ are maximal $F^{\unr}$-split tori of $M$ such that $x_{0} \in \cA(M, T, F^{\unr}) \cap \cA(M, T', F^{\unr})$, there exists $g \in M(F^{\unr})_{x_{0}, 0}$ such that $g T g^{-1} = T'$.
Then by noting that $M(F^{\unr})_{x_{0}, 0}$ acts on $\cA_{x_{0}}$ $(\subseteq \cB(M, F) \subseteq \cB(M, F^{\unr}))$ trivially, the lemma follows from the same argument as in the proof of \cite[Lemma~5.2.1]{2024arXiv240807801A}.
\end{proof}

For $a \in \Phi_{\aff}(G, A_{M})$, let $H_{a}$ denote the affine hyperplane in $\cA_{x_0}$ defined by $H_{a} = \left\{
x \in \cA_{x_0} \mid a(x) = 0
\right\}$.
We define
$
\mathfrak{H} = \left\{
H_{a} \mid a \in \Phi_{\aff}(G, A_{M})
\right\}
$.
We note that 
$
\cA_{\gen} = \cA_{x_0} \smallsetminus
\left(
\bigcup_{H \in \mathfrak{H}} H
\right)
$.
For $x, y \in \cA_{\gen}$, we define the subset $\mathfrak{H}_{x, y}$ of $\mathfrak{H}$ by $\mathfrak{H}_{x, y} = \left\{
H \in \mathfrak{H} \mid \text{$x$ and $y$ are on opposite sides of $H$}
\right\}$.
Let $\mathfrak{H}_{\Krel}$ denote the subset of $\mathfrak{H}$ consisting of the affine hyperplanes that are $\cK$-relevant in the sense of \cite[Definition~3.5.6]{2024arXiv240807801A}, that is, we have $\Theta_{x \mid y} \circ \Theta_{y \mid x}  \not \in \Coeff \cdot \id_{\ind_{K_{x}}^{G(F)}(\rho_{x})}$ for some $x, y \in \cA_{\gen}$ such that $\mathfrak{H}_{x, y} = \{H\}$, where 
$
\Theta_{y \mid x} \colon \ind_{K_{x}}^{G(F)} (\rho_{x}) \rightarrow \ind_{K_{y}}^{G(F)} (\rho_{y})
$
and
$
\Theta_{x \mid y} \colon \ind_{K_{y}}^{G(F)} (\rho_{y}) \rightarrow \ind_{K_{x}}^{G(F)} (\rho_{x})
$
are the intertwining operators defined in \cite[Section~3.5]{2024arXiv240807801A}.
According to \cite[Corollalry~3.5.20]{2024arXiv240807801A}, the set $\mathfrak{H}_{\Krel}$ is preserved by the action of $\Nheart$.
For $H \in \mathfrak{H}$, let $s_{H}$ denote the orthogonal reflection on $\cA_{x_0}$ with respect to $H$.
We define the subgroup $W_{\Krel}$ of the group of affine transformations of $\cA_{x_0}$ by
$W_{\Krel} = \left \langle
s_{H} \mid H \in \mathfrak{H}_{\Krel}
\right \rangle$.
According to \cite[Proposition~5.3.5(1)]{2024arXiv240807801A}, there exists a normal subgroup $\Waff$ of $\Wheart$ such that the action of $\Wheart$ on $\cA_{x_{0}}$ restricts to an isomorphism $\Waff \isoarrow W_{\Krel}$.
We identify $\Waff$ with $W_{\Krel}$ via this isomorphism.
We define a quotient affine space $\cA_{\Krel}$ of $\cA_{x_{0}}$ by 
\[
\cA_{\Krel} = \cA_{x_{0}}/\left(
\bigcap_{H \in \mathfrak{H}_{\Krel}} \ker(D_{A_{M}}(a_{H}))
\right),
\]
where $a_{H}$ denotes an element in $\Phi_{\aff}(G, A_{M})$ such that $H_{a_{H}} = H$, and $D_{A_{M}}(a_{H})$ denotes its gradient.
We identify an affine hyperplane $H \in \mathfrak{H}_{\Krel}$ with its image on $\cA_{\Krel}$.
According to \cite[Proposition~3.7.4]{2024arXiv240807801A}, the group $\Waff$ is the affine Weyl group of an affine root system on the affine space $\cA_{\Krel}$ whose vanishing affine hyperplanes are exactly $\mathfrak{H}_{\Krel}$.
Let $C_{\Krel}$ be the chamber of $\cA_{\Krel}$ with respect to $\mathfrak{H}_{\Krel}$ that contains the image of $x_{0}$ and let $S_{\Krel} \subseteq W(\rho_{M})_{\aff}$ be the subset of simple reflections corresponding to $C_{\Krel}$.
According to \cite[Proposition~3.7.6]{2024arXiv240807801A}, we have $\Wheart = \Wzero \ltimes \Waff$, where $\Wzero$ is the stabilizer of $C_{\Krel}$ in $\Wheart$.

\begin{theorem}[{\cite[Theorem~5.3.6]{2024arXiv240807801A}}]
\label{thm:descriptionofhecke}
We have an isomorphism of $\Coeff$-algebras 
\[
\cH(G(F), (K_{x_0}, \rho_{x_0})) \simeq \Coeff[\Wzero, \mu] \ltimes \cH_{\Coeff}(\Waff, q)
\]
that preserves the anti-involutions on each algebra defined in \cite[Section~3.11]{2024arXiv240807801A}, where $\Coeff[\Wzero, \mu]$ denotes the group algebra of $\Wzero$ twisted by a 2-cocycle $\mu$ on $\Wzero$ introduced in \cite[Notation~3.6.1]{2024arXiv240807801A}, the algebra $\cH_{\Coeff}(\Waff, q)$ is the affine Hecke algebra with $\Coeff$-coefficients attached to the Coxeter system $(\Waff, S_{\Krel})$ and a parameter function $q \colon s \mapsto q_s$ on $S_{\Krel}$, and the meaning of their semi-direct product is explained in \cite[Notation~3.10.8]{2024arXiv240807801A}.
\end{theorem}
The parameter function $q$ is calculated as follows.
\begin{proposition}[{\cite[Theorem~3.18 (ii)]{MR570873}, \cite[Proposition~3.9.1]{2024arXiv240807801A}}]
\label{calculationofqs}
Let $s \in S_{\Krel}$ and $H_{s}$ be the corresponding affine hyperplane in $\cA_{x_{0}}$.
 Let $x, y \in \cA_{\gen}$ such that $\mathfrak{H}_{x, y} = \{H_{s}\}$ and let $h \in H_{s}$ denote the unique point for which $h = x + t \cdot (y - x)$ for some $0 < t < 1$.
 Then we have $q_{s} = q\left(
\ind_{K_{x}}^{K_{h}} (\rho_{x})
\right)$.
\end{proposition}

The main result Theorem~\ref{mainthm:heckealgebraisom} of this paper is that the affine Hecke algebra $\cH_{\Coeff}(\Waff, q)$ is isomorphic to the affine Hecke algebra $\cH_{\Coeff}(W(u_{\rho_{M}})_{\aff}, q_{\theta})$ appearing in the description \cite[Theorem~5.3.6]{2024arXiv240807801A} of the Hecke algebra $\cH\left(
G_{\theta}(F), (K_{\theta, x_{0}}, u_{x_{0}})
\right)$ attached to a unipotent type $(K_{\theta, x_{0}}, u_{x_{0}})$ for a connected reductive group $G_{\theta}$ that splits over $F^{\unr}$, which we will define in the following subsection. 

In the end of this subsection, we give a criterion of $\cK$-relevance that will be used in the proof of Theorem~\ref{mainthm:heckealgebraisom}.
\begin{lemma}
\label{lemma:criterionofrelevance}
Let $H \in \mathfrak{H}$.
Let $x, y \in \cA_{\gen}$ such that $\mathfrak{H}_{x, y} = \{H\}$ and let $h \in H$ denote the unique point for which $h = x + t \cdot (y - x)$ for some $0 < t < 1$.
\begin{enumerate}[(1)]
\item
If there exists an element $n_{H} \in \Nheart \cap K_{h}$ such that the action of $n_{H}$ on $\cA_{x_0}$ agrees with the orthogonal reflection $s_{H}$, then the representation $\ind_{K_{x}}^{K_{h}} (\rho_{x})$ has length two.
Otherwise, the representation $\ind_{K_{x}}^{K_{h}} (\rho_{x})$ is irreducible.
\item
\label{criterionofKrelevance}
The affine hyperplane $H$ is $\cK$-relevant if and only if $q\left(
\ind_{K_{x}}^{K_{h}} (\rho_{x})
\right) > 1$, that is,
the representation $\ind_{K_{x}}^{K_{h}} (\rho_{x})$ has length two and decomposes as $\ind_{K_{x}}^{K_{h}} (\rho_{x}) = \rho_{1} \oplus \rho_{2}$ with $\dim(\rho_{1}) \neq \dim(\rho_{2})$.
\end{enumerate}
\end{lemma}
\begin{proof}

According to \cite[(2.2.6)]{2024arXiv240807801A} and \cite[Proposition~3.4.18]{2024arXiv240807801A}, we have
\[
\dim_{\Coeff}\left(
\End_{K_{h}}\left(
\ind_{K_{x}}^{K_{h}} (\rho_{x})
\right)
\right)
=
\abs{
\left\{
\left(
\Nheart \cap K_{h}
\right)/K_{M}
\right\}
}.
\]
The same argument as in the proof of \cite[Proposition~5.3.5]{2024arXiv240807801A} implies that the group $\left(
\Nheart \cap K_{h}
\right)/K_{M}$ acts on $\cA_{x_{0}}$ properly, and the action of any non-trivial element in $\left(
\Nheart \cap K_{h}
\right)/K_{M}$ on $\cA_{x_0}$ agrees with the orthogonal reflection $s_{H}$.
Thus, we obtain the first claim.

We will prove the second claim. 
Suppose that $H$ is $\cK$-relevant.
Then we have $s_{H} \in \Waff \subseteq \Wheart$.
Hence, \cite[Proposition~5.3.5(1)]{2024arXiv240807801A} and the first claim of the lemma imply that the representation $\ind_{K_{x}}^{K_{h}} (\rho_{x})$ decomposes into a direct sum of irreducible representations $\rho_{1}$ and $\rho_{2}$ of $K_{h}$.
We will prove that $\dim(\rho_{1}) \neq \dim(\rho_{2})$.
If $s_{H} \in S_{\Krel}$, the claim follows from \cite[Proposition~3.8.23, Proposition~3.9.1]{2024arXiv240807801A}.
For general $H$, according to \cite[Chapter~V, Section~3.1, Lemma~2]{MR0240238}, there exists $w \in \Waff$ such that $s_{w(H)} \in S_{\Krel}$.
Since the conditions in \eqref{criterionofKrelevance} are preserved by the action of $\Wheart$, we obtain the claim.

Conversely, suppose that the representation $\ind_{K_{x}}^{K_{h}} (\rho_{x})$ decomposes as $\ind_{K_{x}}^{K_{h}} (\rho_{x}) = \rho_{1} \oplus \rho_{2}$ with $\dim(\rho_{1}) \neq \dim(\rho_{2})$.
According to \cite[Lemma~3.9.2]{2024arXiv240807801A}, there exists $\Phi_{h} \in \End_{K_{h}}\left(
\ind_{K_{x}}^{K_{h}} (\rho_{x})
\right) \simeq \cH(K_{h}, (K_{x}, \rho_{x}))$ with support $K_{x} s_{H} K_{x}$ such that 
\[
\left(
\Phi_{h}
\right)^{2} = (q-1) \cdot \Phi_{h} + q \cdot \id_{\ind_{K_{x}}^{K_{h}} (\rho_{x})},
\]
where $q = \frac{\dim(\rho_1)}{\dim(\rho_2)} \neq 1$.
We regard $\Phi_{h}$ as an element of $\End_{G(F)}\left(
\ind_{K_{x}}^{G(F)} (\rho_{x})
\right)$ by \cite[(2.2.7)]{2024arXiv240807801A}.
According to \cite[Proposition~3.4.18, Lemma~3.5.26]{2024arXiv240807801A}, there exists $c \in \Coeff^{\times}$ such that $\Phi_{h} = c \cdot \Phi_{x, s_{H}}$, where $\Phi_{x, s_{H}}$ is the element of $\End_{G(F)}\left(
\ind_{K_{x}}^{G(F)} (\rho_{x})
\right)$ defined in \cite[Definition~3.5.23]{2024arXiv240807801A}.
According to \cite[Lemma~3.8.9]{2024arXiv240807801A}, which is stated for $H \in \mathfrak{H}_{\Krel}$ but the same proof also works for general $H$, we have 
\[
\Theta_{x \mid y} \circ \Theta_{y \mid x} = c' \cdot \left(
\Phi_{h}
\right)^{2}
= c' \left(
(q-1) \cdot \Phi_{h} + q \cdot \id_{\ind_{K_{x}}^{G(F)} (\rho_{x})}
\right)
\]
 for some $c' \in \Coeff^{\times}$.
Since $q \neq 1$, we obtain that $\Theta_{x \mid y} \circ \Theta_{y \mid x}  \not \in \Coeff \cdot \id_{\ind_{K_{x}}^{G(F)}(\rho_{x})}$.
Thus, the hyperplane $H$ is $\cK$-relevant.
\end{proof}

\subsection{The definition of $G_{\theta}$ and $u_{x_{0}}$}
\label{subsec:defofGthetaandu}

In this subsection, we define the connected reductive group $G_{\theta}$ and the unipotent type $(K_{\theta, x_{0}}, u_{x_{0}})$ for $G_{\theta}$.
After that we will apply \cite[Theorem~5.3.6]{2024arXiv240807801A} to $(K_{\theta, x_{0}}, u_{x_{0}})$ and obtain the description of the corresponding Hecke algebra $\cH\left(
G_{\theta}(F), (K_{\theta, x_{0}}, u_{x_{0}})
\right)$.
In the following subsection, we will prove that the affine Hecke algebras appearing in the descriptions of $\cH(G(F), (K_{x_0}, \rho_{x_0}))$ and $\cH\left(
G_{\theta}(F), (K_{\theta, x_{0}}, u_{x_{0}})
\right)$ are isomorphic.

In this and the following subsections, we write $ \Phi(G, T, F^{\unr}) = \Phi(G, T)$ and $\Phi_{\aff}(G, T, F^{\unr}) = \Phi_{\aff}(G, T)$ for simplicity.
We define a $\Gal(F^{\unr}/F)$-stable subset $\Phi(G, T)_{\theta}$ of $\Phi(G, T)$ by
\[
\Phi(G, T)_{\theta} = \left\{
\alpha \in \Phi(G, T) \mid \theta \circ N_{F' / F} \circ \alpha^{\vee} \restriction_{\cO_{F'}^{\times}} = 1
\right\},
\]
where $F'$ is a finite unramified extension of $F$ such that $T$ splits over $F'$, and $N_{F'/F}$ denotes the norm map $T(F') \rightarrow T(F)$.
We also define 
\[
\Phi_{\aff}(G, T)_{\theta} = \left\{
a \in \Phi_{\aff}(G, T) \mid D_{T}(a) \in \Phi(G, T)_{\theta}
\right\}.
\]
We define the subspace $V^{\theta}$ of $X_{*}(T) \otimes_{\bZ} \bR$ by
$
V^{\theta} = \bigcap_{\alpha \in \Phi(G, T)_{\theta}} \ker(\alpha)
$ and define the affine space $\cA_{\theta}$ by $\cA_{\theta} = \cA(G, T, F^{\unr})/V^{\theta}$.
We write its vector space of translations as $V_{\theta} = \left(
X_{*}(T) \otimes_{\bZ} \bR
\right)/V^{\theta}$.
Then $\Phi(G, T)_{\theta}$ is a root system in $V_{\theta}$ and $\Phi_{\aff}(G, T)_{\theta}$ is an affine root system on $\cA_{\theta}$.

To prove the Hecke algebra isomorphism Theorem~\ref{mainthm:heckealgebraisom}, we want to define $G_{\theta}$ to be a connected reductive group with maximal torus $T$ and affine root system $\Phi_{\aff}(G, T)_{\theta}$.
When $G$ splits over an unramified extension of $F$, we can define such a group $G_{\theta}$.
However, if $G$ is ramified, it can happen that the affine root system $\Phi_{\aff}(G, T)_{\theta}$ does not agree with the affine root system of any connected reductive group with maximal torus $T$.
To resolve this problem, we will ``normalize'' $\Phi(G, T)_{\theta}$ and $\Phi_{\aff}(G, T)_{\theta}$ and define a reduced root system $\Phi(G, T)^{\normal}_{\theta}$ in $V_{\theta}$ and an affine root system $\Phi_{\aff}(G, T)^{\normal}_{\theta}$ on $\cA_{\theta}$ following \cite[Chapter~VI, Section~2.5, Proposition~8]{MR0240238}.
We fix a special point $x_{s}$ in $\cA_{\theta}$ for the affine root system $\Phi_{\aff}(G, T)_{\theta}$ and write 
\[
\Phi_{\aff}(G, T)_{\theta, x_{s}} = \left\{
a \in \Phi_{\aff}(G, T)_{\theta} \mid a(x_{s}) = 0
\right\}.
\]
For $a \in \Phi_{\aff}(G, T)_{\theta, x_{s}}$, let $r_{a} > 0$ be the minimal positive number such that 
\[
\left\{
a + r_{a}, 2(a + r_{a}), \frac{a + r_{a}}{2}
\right\} \cap \Phi_{\aff}(G, T)_{\theta} \neq \emptyset.
\]
We define 
\[
\Phi(G, T)^{\normal}_{\theta} = \left\{
D_{T}(a)/r_{a} \mid a \in \Phi_{\aff}(G, T)_{\theta, x_{s}}
\right\}
\]
and
\[
\Phi_{\aff}(G, T)^{\normal}_{\theta} = \left\{
a/r_{a} + k \mid \alpha \in \Phi_{\aff}(G, T)_{\theta, x_{s}}, k \in \bZ
\right\}.
\]
We note that if $G$ splits over $F^{\unr}$, we have $r_{a} = 1$ for all $a \in \Phi_{\aff}(G, T)_{\theta, x_{s}}$, and we also have $\Phi(G, T)^{\normal}_{\theta} = \Phi(G, T)_{\theta}$ and $\Phi_{\aff}(G, T)^{\normal}_{\theta} = \Phi_{\aff}(G, T)_{\theta}$.
In general, according to \cite[Chapter~VI, Section~2.5, Proposition~8]{MR0240238}, the set $\Phi(G, T)^{\normal}_{\theta}$ is a reduced root system in $V_{\theta}$, and the set $\Phi_{\aff}(G, T)^{\normal}_{\theta}$ is an affine root system on $\cA_{\theta}$ that agrees with $\Phi_{\aff}(G, T)_{\theta}$ up to scalar multiples in the sense that for any $a \in \Phi_{\aff}(G, T)_{\theta}$, we have $\bQ a \cap \Phi_{\aff}^{\normal}(G, T)_{\theta} \neq \emptyset$ and vice versa.

We fix a basis $\Delta$ of $\Phi(G, T)^{\normal}_{\theta}$.
Let $W_{\theta}$ be the Weyl group of $\Phi(G, T)^{\normal}_{\theta}$ and $W^{\ext}_{\theta}$ be the extended affine Weyl group of $\Phi_{\aff}(G, T)^{\normal}_{\theta}$ in the sense of \cite[Definition~1.3.71]{KalethaPrasad}.
For $w \in W^{\ext}_{\theta}$, let $D(w) \in W_{\theta}$ be its derivative.

The following lemma is inspired by a discussion with David Schwein.
\begin{lemma}
Let $\gamma \in \Gal(F^{\unr}/F)$.
Then there exists a unique element $w_{\gamma} \in W^{\ext}_{\theta}$ such that $\gamma(\Delta) = D(w_{\gamma})(\Delta)$ and $\gamma(x_{s}) = w_{\gamma}(x_{s})$.
\end{lemma}
\begin{proof}
Let $W^{\ext}_{\theta, x_{s}}$ be the stabilizer of $x_{s}$ in $W^{\ext}_{\theta}$ and $T^{\ext}_{\theta} = W^{\ext}_{\theta} \cap V_{\theta}$ be the subgroup of translations in $W^{\ext}_{\theta}$.
Since $x_{s}$ is a special point, we have $W^{\ext}_{\theta} = W^{\ext}_{\theta, x_{s}} \ltimes T^{\ext}_{\theta}$, and
the map $w \mapsto D(w)$ restricts to an isomorphism $W^{\ext}_{\theta, x_{s}} \isoarrow W_{\theta}$.
Then the lemma follows from the facts that there exist a unique $w_{0} \in W_{\theta}$ such that $\gamma(\Delta) = w_{0}(\Delta)$ and a unique $t \in T^{\ext}_{\theta}$ such that $\gamma(x_{s}) = t(x_{s})$ by \cite[Lemma~1.3.74]{KalethaPrasad}.
\end{proof}

We define a subset $(\Phi(G, T)^{\normal}_{\theta})^{\vee}$ of $X_{*}(T) \otimes_{\bZ} \bQ$ by
\[
(\Phi(G, T)^{\normal}_{\theta})^{\vee} = \left\{
r_{a} D_{T}(a)^{\vee} \mid a \in \Phi_{\aff}(G, T)_{\theta, x_{s}}
\right\}
\]
and write
\[
X_{*}(T)_{\theta} = \langle X_{*}(T), (\Phi(G, T)^{\normal}_{\theta})^{\vee} \rangle \subseteq X_{*}(T) \otimes_{\bZ} \bQ.
\]
We also define a free abelian group $X^{*}(T)_{\theta} \subseteq X^{*}(T)$ by
\[
X^{*}(T)_{\theta} = \Hom_{\bZ}\left(
X_{*}(T)_{\theta}, \bZ
\right).
\]
Let $T_{\theta}$ be an $F^{\unr}$-split $F$-torus such that $X^{*}(T_{\theta}) = X^{*}(T)_{\theta}$ and let $i_{\theta} \colon T \to T_{\theta}$ be the isogeny corresponding to the inclusion $X^{*}(T)_{\theta} \subseteq X^{*}(T)$.
Then the isogeny $i_{\theta}$ induces an identification
\[
X_{*}(T) \subseteq X_{*}(T)_{\theta} = X_{*}(T_{\theta}).
\]
\begin{remark}
In what follows, we define the group $G_{\theta}$ using the torus $T_{\theta}$.
More generally, we may replace the torus $T_{\theta}$ with any $F^{\unr}$-split $F$-torus $T'_{\theta}$ equipped with an isogeny $i'_{\theta} \colon T \to T'_{\theta}$ such that
\[
\Phi(G, T)^{\normal}_{\theta} \subseteq X^{*}(T'_{\theta})
\qquad
\text{and}
\qquad
(\Phi(G, T)^{\normal}_{\theta})^{\vee} \subseteq X_{*}(T'_{\theta}),
\]
where we regard $X^{*}(T'_{\theta})$ (resp.\ $X_{*}(T'_{\theta})$) as a subset of $X^{*}(T) \otimes_{\bZ} \bQ$ (resp.\ $X_{*}(T) \otimes_{\bZ} \bQ$) via the isogeny $i'_{\theta}$.
\end{remark}
Let $G_{\theta, \qs}$ denote the connected reductive group that is defined and splits over $F^{\unr}$ with maximal torus $(T_{\theta})_{F^{\unr}}$ and root datum
\[
\Psi_{\theta} =
\left(
X^{*}(T_{\theta}), \Phi(G_{\theta, \qs}, T_{\theta}) := \Phi(G, T)^{\normal}_{\theta}, X_{*}(T_{\theta}), (\Phi(G, T)^{\normal}_{\theta})^{\vee}
\right).
\]
We fix a pinning $\{X_{\alpha}\}_{\alpha \in \Delta}$ of $G_{\theta, \qs}$ and let $o \in \cA(G_{\theta, \qs}, (T_{\theta})_{F^{\unr}}, F^{\unr})$ be the corresponding Chevalley valuation.
Then we obtain an isomorphism of affine spaces
\[
f \colon \cA^{\red}(G_{\theta, \qs}, (T_{\theta})_{F^{\unr}}, F^{\unr}) \isoarrow \cA_{\theta}
\]
defined by $f(o + v) = x_{s} + v$ for $v \in V_{\theta}$.
Under this identification, we have 
\[
\Phi_{\aff}(G_{\theta, \qs}, (T_{\theta})_{F^{\unr}}, F^{\unr}) = \Phi_{\aff}(G, T)^{\normal}_{\theta}.
\]
The map $\gamma \mapsto D(w_{\gamma})^{-1} \circ \gamma$ defines an action of $\Gal(F^{\unr}/F)$ on $\Psi_{\theta}$ that preserves $\Delta$.
This action induces an action of $\Gal(F^{\unr}/F)$ on $G_{\theta, \qs}$ preserving the pinning $\{X_{\alpha}\}_{\alpha \in \Delta}$, which gives an $F$-structure on $G_{\theta, \qs}$.
By using this $F$-structure, we regard $G_{\theta, \qs}$ as a quasi-split connected reductive group over $F$.
We note that the $F$-structure on $(T_{\theta})_{F^{\unr}}$ given by the $F$-structure on $G_{\theta, \qs}$ is different from the original $F$-structure on $T_{\theta}$.
To distinguish them, we write $T_{\qs}$ for this maximal $F$-torus of $G_{\theta, \qs}$.
The $F$-structures on $G_{\theta, \qs}$ and $T_{\qs}$ induce an action of $\Gal(F^{\unr}/F)$ on the reduced apartment $\cA^{\red}(G_{\theta, \qs}, T_{\qs}, F^{\unr})$.
We note that the isomorphism $f$ is not $\Gal(F^{\unr}/F)$-equivariant.
Instead, for $\gamma \in \Gal(F^{\unr}/F)$ and $x \in \cA^{\red}(G_{\theta, \qs}, T_{\qs}, F^{\unr})$, we have
\begin{equation}
\label{innertwist}
\left(
f^{-1} \circ (\gamma \circ f \circ \gamma^{-1})
\right)(x) = w_{\gamma}(x),
\end{equation}
where we identify $W_{\theta}^{\ext}$ with the extended affine Weyl group of $\Phi_{\aff}(G_{\theta, \qs}, T_{\qs}, F^{\unr})$ ignoring Galois actions.
Let $G_{\theta, \qs, \ad}$ denote the adjoint group of $G_{\theta, \qs}$ and $T_{\ad}$ be the image of $T_{\qs}$ in $G_{\theta, \qs, \ad}$.
According to \cite[Proposition~6.6.2]{KalethaPrasad}, the action of $N_{G_{\theta, \qs, \ad}}(T_{\ad})(F^{\unr})$ on $\cA^{\red}(G_{\theta, \qs}, T_{\qs}, F^{\unr})$ induces an isomorphism
\[
N_{G_{\theta, \qs, \ad}}(T_{\ad})(F^{\unr})/T_{\ad}(F^{\unr})_{0} \isoarrow W_{\theta}^{\ext}.
\]
Hence, \eqref{innertwist} gives us a 1-cocycle 
\[
\Gal(F^{\unr}/F) \rightarrow N_{G_{\theta, \qs, \ad}}(T_{\ad})(F^{\unr})/T_{\ad}(F^{\unr})_{0}
\]
given by $\gamma \mapsto w_{\gamma}$.
According to \cite[Lemma~2.3.1]{MR2480618} and \cite[Lemma~3.1.7]{kal19}, we have 
\[
H^1(\Gal(F^{\unr}/F), T_{\ad}(F^{\unr})_{0}) =
H^2(\Gal(F^{\unr}/F), T_{\ad}(F^{\unr})_{0}) = \{1\}.
\]
Thus, the 1-cocycle lifts to a 1-cocycle 
\[
\Gal(F^{\unr}/F) \rightarrow N_{G_{\theta, \qs, \ad}}(T_{\ad})(F^{\unr}), \, \gamma \mapsto n_{\gamma},
\]
and its cohomology class is uniquely determined by the cohomology class of the 1-cocycle $\gamma \mapsto w_{\gamma}$.
Let $\psi \colon G_{\theta, \qs} \rightarrow G_{\theta}$ be the inner twist of $G_{\theta, \qs}$ such that $\left(
\psi^{-1} \circ (\gamma \circ \psi \circ \gamma^{-1})
\right)(g) = n_{\gamma} g n_{\gamma}^{-1}$ for all $\gamma \in \Gal(F^{\unr}/F)$ and $g \in G_{\theta, \qs}(F^{\unr})$.
We also use the same notation $\psi$ for the induced map $\cA(G_{\theta, \qs}, T_{\qs}, F^{\unr}) \rightarrow \cA(G_{\theta}, \psi(T_{\qs}), F^{\unr})$.
By the definition of $G_{\theta}$, the maximal torus $\psi(T_{\qs})$ of $G_{\theta}$ is isomorphic to $T_{\theta}$ as $F$-tori, and the isomorphism
\[
f \circ \psi^{-1} \colon \cA^{\red}(G_{\theta}, \psi(T_{\qs}), F^{\unr}) \isoarrow \cA_{\theta}
\]
is $\Gal(F^{\unr}/F)$-equivariant.
We identify $\gamma(T_{\qs})$ and  $\cA^{\red}(G_{\theta}, \psi(T_{\qs}), F^{\unr})$ with $T_{\theta}$ and $\cA_{\theta}$ via these isomorphisms, respectively.
Then we obtain from the definition of $G_{\theta}$ that $\Phi(G_{\theta}, T_{\theta}) = \Phi(G, T)^{\normal}_{\theta}$ and $\Phi_{\aff}(G_{\theta}, T_{\theta}) = \Phi_{\aff}(G, T)^{\normal}_{\theta}$ as $\Gal(F^{\unr}/F)$-sets.
We also identify $\cA(G_{\theta}, T_{\theta}, F^{\unr})$ with $\cA(G, T, F^{\unr})$ as follows.
We fix a lift $\widetilde{o} \in \cA(G_{\theta}, T_{\theta}, F^{\unr})$ of $\psi(o) \in \cA^{\red}(G_{\theta}, T_{\theta}, F^{\unr})$ and $\widetilde{x}_{s} \in \cA(G, T, F^{\unr})$ of $x_{s} \in \cA_{\theta}$.
For $\gamma \in \Gal(F^{\unr}/F)$, let $z_{\gamma} \in X_{*}(Z(G_{\theta})^{\circ}) \otimes_{\bZ} \bR$ be the element such that
\[
\gamma(\widetilde{o}) - \widetilde{o} = \gamma(\widetilde{x}_{s}) - \widetilde{x}_{s} + z_{\gamma},
\]
where $Z(G_{\theta})^{\circ}$ denotes the identity component of $Z(G_{\theta})$.
Since $H^{1}(\Gal(F^{\unr}/F), X_{*}(Z(G_{\theta})^{\circ}) \otimes_{\bZ} \bR) = \{1\}$, we can take $z \in X_{*}(Z(G_{\theta})^{\circ}) \otimes_{\bZ} \bR$ such that $\gamma(z) - z = z_{\gamma}$ for all $\gamma \in \Gal(F^{\unr}/F)$, 
Then the isomorphism
\[
\widetilde{f \circ \psi^{-1}} \colon \cA(G_{\theta}, T_{\theta}, F^{\unr}) \isoarrow \cA(G, T, F^{\unr})
\]
defined by $\widetilde{o} + v \mapsto \widetilde{x}_{s} + z + v$ for $v \in X_{*}(T) \otimes_{\bZ} \bR = X_{*}(T_{\theta}) \otimes_{\bZ} \bR$ is $\Gal(F^{\unr}/F)$-equivariant, and the following diagram commutes:
\[
\xymatrix@R+1pc@C+1pc{
\cA(G_{\theta}, T_{\theta}, F^{\unr})
	\ar[d]
	\ar[r]^-{\widetilde{f \circ \psi^{-1}}}
	\ar@{}[dr]|\circlearrowleft
& \cA(G, T, F^{\unr})
	\ar[d] \\
\cA^{\red}(G_{\theta}, T_{\theta}, F^{\unr})
	\ar[r]^-{f \circ \psi^{-1}}
& \cA_{\theta}.
}
\]
We identify $\cA(G_{\theta}, T_{\theta}, F^{\unr})$ with $\cA(G, T, F^{\unr})$ via $\widetilde{f \circ \psi^{-1}}$.
Under this identification, we regard $x_{0} \in \cA_{x_{0}} \subseteq \cA(G, T, F^{\unr})$ as an element of $\cA(G_{\theta}, T_{\theta}, F^{\unr})$.

Let $M_{\theta}$ denote the Levi subgroup of $G_{\theta}$ with maximal torus $T_{\theta}$ and root system
\[
\Phi(M_{\theta}, T_{\theta}) = \Phi(G_{\theta}, T_{\theta}) \cap \bQ \Phi(M, T) \subseteq X^{*}(T) \otimes_{\bZ} \bQ.
\]
We fix an admissible embedding $\iota_{\theta} \colon \cB(M_{\theta}, F) \hookrightarrow \cB(G_{\theta}, F)$ and regard $x_{0} \in \cB(M_{\theta}, F)$.
Then the point $[x_{0}]_{M_{\theta}}$ is a vertex of $\cB^{\red}(M_{\theta}, F)$, and the embedding $\iota_{\theta}$ is $0$-generic relative to $x_{0}$.
Since $T$ is an elliptic maximal torus of $M$, we have 
\[
X_{*}(A_{M_{\theta}}) \otimes_{\bZ} \bR = X_{*}(A_{M}) \otimes_{\bZ} \bR
\]
as subspaces of $X_{*}(T_{\theta}) \otimes_{\bZ} \bR = X_{*}(T) \otimes_{\bZ} \bR $.
Hence, we have $x_{0} + (X_{*}(A_{M_{\theta}}) \otimes_{\bZ} \bR) = x_{0} + (X_{*}(A_{M}) \otimes_{\bZ} \bR) = \cA_{x_{0}}$ as subsets of $\cA(G_{\theta}, T_{\theta}, F^{\unr}) = \cA(G, T, F^{\unr})$.

\begin{proposition}
\label{prop:thetaxvsxtheta}
Let $x \in \cA(G_{\theta}, T_{\theta}, F^{\unr}) = \cA(G, T, F^{\unr})$.
Then the pairs $((\sfG_{\theta})^{\circ}_{x}, (\sfM_{\theta})^{\circ}_{x})$ and $((\sfG^{\circ}_{x})^{*}_{\theta}, (\sfM^{\circ}_{x})^{*}_{\theta})$ are adjointly isomorphic up to taking duals in the sense of Definition~\ref{def:almostisomforgroups}, where $(\sfG_{\theta})^{\circ}_{x}$ and $(\sfM_{\theta})^{\circ}_{x}$ denote the reductive quotients of the special fibers of the connected parahoric group schemes associated to the point $x \in \cB(M_{\theta}, F) \subseteq \cB(G_{\theta}, F)$, and $(\sfG^{\circ}_{x})^{*}_{\theta}$ and $(\sfM^{\circ}_{x})^{*}_{\theta}$ denote the groups defined as in Section~\ref{subsec:statementofmainresult} for the connected reductive groups $\sfG^{\circ}_{x}$ and $\sfM^{\circ}_{x}$ over $\ff$.
\end{proposition}
\begin{proof}
It suffices to show that the pairs $((\sfG_{\theta})^{\circ}_{x}, (\sfM_{\theta})^{\circ}_{x})$ and $((\sfG^{\circ}_{x})_{\theta}, (\sfM^{\circ}_{x})_{\theta})$ are adjointly isomorphic up to taking duals.
According to \cite[Proposition~9.4.23]{KalethaPrasad}, we have
\[
\Phi((\sfG_{\theta})^{\circ}_{x}, \sfT^{\circ}) = \left\{
D_{T}(a) \mid a \in \Phi_{\aff}(G_{\theta}, T_{\theta}), \, a(x) = 0
\right\}
\]
and
\begin{align*}
\Phi((\sfG^{\circ}_{x})_{\theta}, \sfT^{\circ})
&=
\left\{
\alpha \in \Phi(\sfG^{\circ}_{x}, \sfT^{\circ}) \mid \theta \circ N_{\ff'/\ff} \circ \alpha^{\vee} = 1
\right\} \\
&=
\left\{
D_{T}(a) \mid a \in \Phi_{\aff}(G, T), \, a(x) = 0
\right\} \cap \left\{
\alpha \in \Phi(G, T) \mid \theta \circ N_{F' / F} \circ \alpha^{\vee} \restriction_{\cO_{F'}^{\times}} = 1
\right\} \\
&=
\left\{
D_{T}(a) \mid a \in \Phi_{\aff}(G, T)_{\theta}, \, a(x) = 0
\right\}.
\end{align*}
Since the affine root system $\Phi_{\aff}(G_{\theta}, T_{\theta}) = \Phi_{\aff}(G, T)^{\normal}_{\theta}$ agrees with the affine root system $\Phi_{\aff}(G, T)_{\theta}$ up to scalar multiples, the root systems $\Phi((\sfG_{\theta})^{\circ}_{x}, \sfT^{\circ})$ and $\Phi((\sfG^{\circ}_{x})_{\theta}, \sfT^{\circ})$ satisfy the assumptions of Lemma~\ref{lemmacriterionforalmostisomorphic}.
Thus, the proposition follows from Lemma~\ref{lemmacriterionforalmostisomorphic}.
\end{proof}

\begin{remark}
\label{remarkwhyuptotakingduals}
If $G$ splits over $F^{\unr}$, we have $(\sfG_{\theta})^{\circ}_{x} \simeq (\sfG^{\circ}_{x})_{\theta}$ since the affine root system $\Phi_{\aff}(G_{\theta}, T_{\theta}) = \Phi_{\aff}(G, T)^{\normal}_{\theta}$ agrees with $\Phi_{\aff}(G, T)_{\theta}$.
However, if $G$ is ramified, the affine root system $\Phi_{\aff}(G_{\theta}, T_{\theta}) $ agrees with the affine root system $\Phi_{\aff}(G, T)_{\theta}$ only up to scalar multiples.
This is the reason why we need the generalization Corollary~\ref{corollary:reductiontounipdualver} of Proposition~\ref{prop:reductiontounipotentqparameters} to prove the main theorem.
\end{remark}

Let $u^{*}_{\sfM^{\circ}_{x_0}}$ denote the unipotent, cuspidal representation of $(\sfM^{\circ}_{x_0})^{*}_{\theta}(\ff)$ attached to $\rho_{M, 0}$ as in Section~\ref{section:finitegroupcase}.
We define the unipotent, cuspidal representation $u_{\sfM^{\circ}_{x_0}}$ of $(\sfM_{\theta})^{\circ}_{x_0}(\ff)$ as the unique unipotent, cuspidal representation that is adjointly isomorphic to $u^{*}_{\sfM^{\circ}_{x_0}}$ up to taking duals.
We write $u_{M, 0}$ for the inflation of $u_{\sfM^{\circ}_{x_0}}$ to $M_{\theta}(F)_{x_{0}, 0}$.
Let $u_{M, \dc}$ be an irreducible representation of $M(F)_{x_{0}}$ such that $u_{M, \dc} \restriction_{M(F)_{x_{0}, 0}}$ contains $u_{M, 0}$.
Let $(K_{\theta, M}, u_{\rho_{M}})$ be either $(M_{\theta}(F)_{x_0, 0}, u_{M, 0})$ or $(M_{\theta}(F)_{x_0}, u_{M, \dc})$.

Now, we will apply \cite[Section~5]{2024arXiv240807801A} to $(G_{\theta}, M_{\theta}, K_{\theta, M}, u_{\rho_{M}})$.
We define 
\[
\Wthetaheart = \left(
N_{G_{\theta}(F)}(u_{\rho_{M}}) \cap N_{G_{\theta}}(M_{\theta})(F)_{[x_0]_{M_{\theta}}}
\right)/ K_{\theta, M}
\]
and
\begin{equation}
\label{descriptionofPhiafftheta}
\Phi_{\aff}(G_{\theta}, A_{M}) = 
\left\{
a \restriction_{\cA_{x_{0}}} \mid a \in \Phi_{\aff}(G_{\theta}, T_{\theta}) \smallsetminus \Phi_{\aff}(M_{\theta}, T_{\theta})
\right\}.
\end{equation}

We define the subset $\mathfrak{H}_{\theta}$ of $\mathfrak{H}$ by
$
\mathfrak{H}_{\theta} = \left\{
H_{a} \mid a \in \Phi_{\aff}(G_{\theta}, A_{M})
\right\}
$, where $H_{a}$ denotes the affine hyperplane in $\cA_{x_0}$ defined by $H_{a} = \left\{
x \in \cA_{x_0} \mid a(x) = 0
\right\}$.
Since the affine root system $\Phi_{\aff}(G_{\theta}, T_{\theta}) = \Phi_{\aff}(G, T)^{\normal}_{\theta}$ agrees with the affine root system $\Phi_{\aff}(G, T)_{\theta} \subseteq \Phi_{\aff}(G, T)$ up to scalar multiples, Lemma~\ref{lemma:descriptionofphiaffGaM} and \eqref{descriptionofPhiafftheta} imply that $\mathfrak{H}_{\theta} \subseteq \mathfrak{H}$.
We define the subset $\cA_{\theta-\gen}$ of $\cA_{x_0}$ by
$
\cA_{\theta-\gen} = \cA_{x_0} \smallsetminus
\left(
\bigcup_{H \in \mathfrak{H}_{\theta}} H
\right)
$.
We note that $\cA_{\gen} \subseteq \cA_{\theta-\gen}$.
For $x \in \cA_{x_{0}}$, we define $K_{\theta, x} = K_{\theta, M} \cdot G_{\theta}(F)_{x, 0}$ and $K_{\theta, x, +} = G_{\theta}(F)_{x, 0+}$.
If $x \in \cA_{\theta-\gen}$, we define an irreducible representation $u_{x}$ of $K_{x, \theta}$ as the composition of $u_{\rho_{M}}$ with the inverse of the isomorphism $K_{\theta, M}/M_{\theta}(F)_{x, 0+} \isoarrow K_{\theta, x}/K_{\theta, x, +}$ that comes from the inclusion $K_{\theta, M} \subseteq K_{\theta, x}$.
We write 
$\cK_{\theta} = \left\{
(K_{\theta, x}, K_{\theta, x, +}, u_{x})
\right\}_{x \in \cA_{\theta-\gen}}$.
Let $\mathfrak{H}_{\Kthetarel}$ denote the subset of $\mathfrak{H}$ consisting of the affine hyperplanes that are $\cK_{\theta}$-relevant in the sense of \cite[Definition~3.5.6]{2024arXiv240807801A}.
We define $W_{\Kthetarel} = \left \langle
s_{H} \mid H \in \mathfrak{H}_{\Kthetarel}
\right \rangle$.
According to \cite[Proposition~3.7.4, Proposition~5.3.5(1)]{2024arXiv240807801A}, there exists a normal subgroup $\Wthetaaff$ of $\Wthetaheart$ such that the action of $\Wthetaheart$ on $\cA_{x_{0}}$ restricts to an isomorphism $\Wthetaaff \simeq W_{\Kthetarel}$, and $W_{\Kthetarel}$ is the affine Weyl group of an affine root system on the quotient affine space $\cA_{\Kthetarel}$ of $\cA_{x_{0}}$ defined by
\[
\cA_{\Kthetarel} = \cA_{x_{0}}/\left(
\bigcap_{H \in \mathfrak{H}_{\Kthetarel}} \ker(D_{A_{M}}(a_{H}))
\right).
\]
Let $S_{\Kthetarel} \subseteq \Wthetaaff$ denote the subset of simple reflections corresponding to the chamber containing the image of $x_{0}$ in $\cA_{\Kthetarel}$.

Applying \cite[Theorem~5.3.6]{2024arXiv240807801A} and \cite[Proposition~3.9.1]{2024arXiv240807801A} to $(G_{\theta}, M_{\theta}, u_{\rho_{M}}, \cK_{\theta})$, we obtain the following explicit description of $\cH\left(
G_{\theta}(F), (K_{\theta, x_{0}}, u_{x_{0}})
\right)$.
\begin{theorem}[{\cite[Theorem~5.3.6, Proposition~3.9.1]{2024arXiv240807801A}}]
\label{thm:descriptionofhecketheta}
We have an isomorphism of $\Coeff$-algebras 
\[
\cH\left(
G_{\theta}(F), (K_{\theta, x_{0}}, u_{x_{0}})
\right)
\simeq
\Coeff[\Omega(u_{\rho_{M}}), \mu_{\theta}] \ltimes \cH_{\Coeff}(W(u_{\rho_{M}})_{\aff}, q_{\theta})
\]
that preserves the anti-involutions on each algebra defined in \cite[Section~3.11]{2024arXiv240807801A} for a $2$-cocycle $\mu_{\theta}$ on a subgroup $\Omega(u_{\rho_{M}})$ of $\Wthetaheart$ and a parameter function $q_{\theta} \colon s \mapsto q_{\theta, s}$ on $
S_{\Kthetarel}$.
For $s \in S_{\Kthetarel}$, we have $q_{\theta, s} = q\left(
\ind_{K_{\theta, x}}^{K_{\theta, h}} (\rho_{\theta, x})
\right)$, where $x, y \in \cA_{\gen}$ are points such that $\mathfrak{H}_{x, y} = \{H_{s}\}$, and $h \in H_{s}$ is the unique point for which $h = x + t \cdot (y - x)$ for some $0 < t < 1$.
\end{theorem}

\subsection{An isomorphism of affine Hecke algebras}
\label{subsection:An isomorphism of affine Hecke algebras}

In this subsection, we prove the main theorem of this paper:

\begin{theorem}
\label{mainthm:heckealgebraisom}
We have $\mathfrak{H}_{\Krel} = \mathfrak{H}_{\Kthetarel}$ and 
$
W_{\Krel} = W_{\Kthetarel}
$.
Moreover, for $s \in S_{\Krel} = S_{\Kthetarel}$, we have $q_{s} = q_{\theta, s}$.
Thus, we have 
\[
\cH_{\Coeff}(\Waff, q) \simeq \cH_{\Coeff}(\Wthetaaff, q_{\theta}).
\]
\end{theorem}
\begin{proof}
\addtocounter{equation}{-1}
\begin{subequations}
Let $H \in \mathfrak{H}$.
We fix $x, y \in \cA_{\gen}$ such that $\mathfrak{H}_{x, y} = \{H\}$, and let $h \in H$ denote the unique point for which $h = x + t \cdot (y - x)$ for some $0 < t < 1$.
We denote by $\sfP_{x}$ the parabolic subgroup of $\sfG^{\circ}_{h}$ such that $\sfP_{x}(\ff) = G(F)_{x, 0} \cdot G(F)_{h, 0+} / G(F)_{h, 0+}$.
Then the group $\sfG^{\circ}_{x} = \sfM^{\circ}_{x} = \sfM^{\circ}_{h}=  \sfM^{\circ}_{x_{0}}$ is a Levi factor of $\sfP_{x}$.
Applying Proposition~\ref{prop:comparisonofqparameters} to $\sfG' = K_{h}/G(F)_{h, 0+}$, $\sfG = \sfG^{\circ}_{h}(\ff)$, $\sfP' = K_{x} \cdot G(F)_{h, 0+}/G(F)_{h, 0+}$, $\rho' = \rho_{x}$, and $\rho = \rho_{M, 0}$, we obtain that 
\begin{equation}
\label{reductiontoparahoric}
q\left(
\ind_{K_{x}}^{K_{h}} (\rho_{x}) 
\right)
=
q\left(
\ind_{\sfP_{x}(\ff)}^{\sfG^{\circ}_{h}(\ff)} \left(
\rho_{M, 0}
\right)
\right)
=
q\left(
R_{\sfM^{\circ}_{h}}^{\sfG^{\circ}_{h}} \left(
\rho_{M, 0}
\right)
\right).
\end{equation}
Similarly, applying Proposition~\ref{prop:comparisonofqparameters} to $\sfG' = K_{\theta, h}/G_{\theta}(F)_{h, 0+}$, $\sfG = (\sfG_{\theta})^{\circ}_{h}(\ff)$, $\sfP' = K_{\theta, x} \cdot G_{\theta}(F)_{h, 0+}/G_{\theta}(F)_{h, 0+}$, $\rho' = u_{x}$, and $\rho = u_{\sfM^{\circ}_{x_0}}$, we obtain that
\begin{equation}
\label{reductiontoparahorictheta}
q\left(
\ind_{K_{\theta, x}}^{K_{\theta, h}} (u_{x}) 
\right)
=
q\left(
\ind_{(\sfP_{\theta})_{x}(\ff)}^{(\sfG_{\theta})^{\circ}_{h}(\ff)} (u_{\sfM^{\circ}_{x_{0}}})
\right)
=
q\left(
R_{(\sfM_{\theta})^{\circ}_{h}}^{(\sfG_{\theta})^{\circ}_{h}} (u_{\sfM^{\circ}_{x_{0}}})
\right),
\end{equation}
where $(\sfP_{\theta})_{x}$ is the parabolic subgroup of $(\sfG_{\theta})^{\circ}_{h}$ such that $(\sfP_{\theta})_{x}(\ff) = G_{\theta}(F)_{x, 0} \cdot G_{\theta}(F)_{h, 0+}/G_{\theta}(F)_{h, 0+}$.
According to Proposition~\ref{prop:thetaxvsxtheta}, the pairs $((\sfG_{\theta})^{\circ}_{h}, (\sfM_{\theta})^{\circ}_{h})$ and $((\sfG^{\circ}_{h})^{*}_{\theta}, (\sfM^{\circ}_{h})^{*}_{\theta})$ are adjointly isomorphic up to taking duals, and we have $u_{\sfM^{\circ}_{x_{0}}} \sim_{\dual} u^{*}_{\sfM^{\circ}_{x_{0}}}$ by definition.
Thus, we obtain from Corollary~\ref{corollary:reductiontounipdualver} that
\[
q\left(
R_{\sfM^{\circ}_{h}}^{\sfG^{\circ}_{h}} \left(
\rho_{M, 0}
\right)
\right)
=
q\left(
R_{(\sfM_{\theta})^{\circ}_{h}}^{(\sfG_{\theta})^{\circ}_{h}} (u_{\sfM^{\circ}_{x_{0}}})
\right).
\]
Combining this with \eqref{reductiontoparahoric} and \eqref{reductiontoparahorictheta}, we obtain that
\begin{equation}
\label{qparametersreductiontounipotent}
q\left(
\ind_{K_{x}}^{K_{h}} (\rho_{x}) 
\right) = q\left(
\ind_{K_{\theta, x}}^{K_{\theta, h}} (u_{x}) 
\right).
\end{equation}
Hence, Lemma~\ref{lemma:criterionofrelevance} implies that for $H \in \mathfrak{H}_{\theta}$, the conditions that $H \in \mathfrak{H}_{\Krel}$ and $H \in \mathfrak{H}_{\Kthetarel}$ are equivalent.
Thus, to prove $\mathfrak{H}_{\Krel} = \mathfrak{H}_{\Kthetarel}$, it suffices to show that $\mathfrak{H}_{\Krel} \subseteq \mathfrak{H}_{\theta}$.
Let $H \in \mathfrak{H}_{\Krel}$.
Then we obtain from Lemma~\ref{lemma:criterionofrelevance}, Equation~\eqref{reductiontoparahoric}, and Corollary~\ref{corollary:reductiontounipdualver} that 
\[
1 \neq q\left(
\ind_{K_{x}}^{K_{h}} (\rho_{x}) 
\right) 
=
q\left(
R_{\sfM^{\circ}_{h}}^{\sfG^{\circ}_{h}} \left(
\rho_{M, 0}
\right)
\right)
=
q\left(
R_{(\sfM_{\theta})^{\circ}_{h}}^{(\sfG_{\theta})^{\circ}_{h}} (u_{\sfM^{\circ}_{x_{0}}})
\right).
\]

In particular, we have $(\sfM_{\theta})^{\circ}_{h} \subsetneq (\sfG_{\theta})^{\circ}_{h}$.
Hence, there exists $\alpha_{h} \in \Phi((\sfG_{\theta})^{\circ}_{h}, \sfT^{\circ})  \smallsetminus \Phi((\sfM_{\theta})^{\circ}_{h}, \sfT^{\circ})$.
According to \cite[Proposition~9.4.23]{KalethaPrasad}, we have
\[
\Phi((\sfG_{\theta})^{\circ}_{h}, \sfT^{\circ}) = \left\{
D_{T}(a) \mid a \in \Phi_{\aff}(G_{\theta}, T_{\theta}), \, a(h) = 0
\right\}.
\]
We take $a'_{h} \in \Phi_{\aff}(G_{\theta}, T_{\theta}) \smallsetminus \Phi_{\aff}(M_{\theta}, T_{\theta})$ such that $D_{T}(a'_{h}) = \alpha_{h}$ and $a'_{h}(h) = 0$.
We write $a_{h} = a'_{h} \restriction_{\cA_{x_{0}}}$. 
Then we have $a_{h} \in \Phi_{\aff}(G_{\theta}, A_{M})$ by \eqref{descriptionofPhiafftheta}, and $h \in H_{a_{h}} \in \mathfrak{H}_{\theta} \subseteq \mathfrak{H}$.
Since $\mathfrak{H}_{x, y} = \{H\}$, the definition of the point $h$ implies that $H$ is the unique affine hyperplane in $\mathfrak{H}$ such that $h \in H$.
Thus, we have $H = H_{a_{h}} \in \mathfrak{H}_{\theta}$.

Finally, combining Equation~\eqref{qparametersreductiontounipotent} with Proposition~\ref{calculationofqs} and Theorem~\ref{thm:descriptionofhecketheta}, we obtain that
\[
q_{s} = q\left(
\ind_{K_{x}}^{K_{h}} (\rho_{x}) 
\right) = q\left(
\ind_{K_{\theta, x}}^{K_{\theta, h}} (u_{x}) 
\right) = q_{\theta, s}
\]
for $s \in S_{\Krel} = S_{\Kthetarel}$.
\end{subequations}
\end{proof}

\subsection{On Hecke algebras for types constructed by Kim and Yu}
\label{subsec:heckealgebraforkimyutypes}

In this subsection, we assume that $p \neq 2$ and the group $G$ splits over a tamely ramified extension of $F$.

Let $\Sigma = \bigl(
(\overrightarrow{G}, M^0), \overrightarrow{r}, (x_{0}, \{\iota\}), (K_{M^0}, \rho_{M^0}), \overrightarrow{\phi}
\bigr)$ be a $G$-datum in the sense of \cite[Definition~4.1.1]{2024arXiv240807805A} (following \cite[7.2]{Kim-Yu}) such that $K_{M^0} = M^0(F)_{x_0, 0}$ or $M^0(F)_{x_0}$.
In \cite[Section~4.1]{2024arXiv240807805A}, we constructed a pair $(K_{x_{0}}, \rho_{x_{0}})$ of a compact, open subgroup $K_{x_{0}}$ of $G(F)$ and an irreducible smooth representation $\rho_{x_{0}}$ of $K_{x_{0}}$ following the construction of Kim and Yu (\cite[Section~7]{Kim-Yu}), which is based on \cite[Section~4]{Yu}, but twisted by the quadratic character from \cite[Section~4]{FKS}.
By \cite[Theorem~7.5]{Kim-Yu} and \cite{MR4357723},
the pair $(K_{x_{0}}, \rho_{x_{0}})$ is an $\fS(\Sigma)$-type for 
a finite subset $\fS(\Sigma)$ of $\IEC(G)$.
Moreover, according to \cite[Theorem~7.12]{Fi-exhaustion} under the assumption that $p$ does not divide the order of the absolute Weyl group of $G$, for any $\fs \in \IEC(G)$, there exists a $G$-datum $\Sigma$ such that $\{\fs\} = \fS(\Sigma)$.

Let $(K_{x_0}^0, \rho^0_{x_0})$ be the depth-zero type constructed from the depth-zero $G^0$-datum $((G^0, M^0), (x_{0}, \iota:\cB(M^0,F) \longrightarrow \cB(G^0,F)), (K_{M^0}, \rho_{M^0}))$, and we attach $(G^0_{\theta}, M^{0}_{\theta}, u_{\rho_{M^0}}, \cK^0_{\theta})$ to it as in Section~\ref{subsection:An isomorphism of affine Hecke algebras}.

\begin{theorem}
\label{mainthm+AFMO}
We have isomorphisms of $\Coeff$-algebras 
\begin{align*}
\cH(G(F), (K_{x_0}, \rho_{x_0})) &\simeq \cH(G^{0}(F), (K^0_{x_0}, \rho^{0}_{x_{0}})) \\
&\simeq \Coeff[\Wzeroz, \mu^0] \ltimes \cH_{\Coeff}(\Waffz, q) \\
&= \Coeff[\Wzeroz, \mu^0] \ltimes \cH_{\Coeff}(\Wthetaaffz, q_{\theta})
\end{align*}
that preserve the anti-involutions
on each algebra defined in \cite[Section~3.11]{2024arXiv240807801A}.
\end{theorem}
\begin{proof}
The theorem follows from \cite[Theorem~4.4.1]{2024arXiv240807805A} and Theorem~\ref{mainthm:heckealgebraisom}.
\end{proof}

\subsection{On Lusztig's conjecture about parameters}
\label{subsection:Lusztig'sconjecture}

In the previous subsection, we proved that the affine Hecke algebra appearing in the description of the Hecke algebra attached to a type constructed by Kim and Yu is isomorphic to the affine Hecke algebra attached to a unipotent type.
Under the assumption that $G$ splits over a tamely ramified extension of $F$, and $p$ does not divide the order of the absolute Weyl group of $G$, this result establishes a version of Lusztig's conjecture \cite[\S1(a)]{2020arXiv200608535L} that we interpret in a manner compatible with our framework:

\begin{conjecture}[{\cite[\S1(a)]{2020arXiv200608535L}}]
\label{conjecturebyLusztig}
For any Bernstein block $\Rep^\fs(G(F))$ of $G$, we have an equivalence of categories 
\[
\Rep^\fs(G(F)) \simeq \Mod\cH^{\fs},
\]
where $\cH^{\fs}$ is an extension of an affine Hecke algebra whose parameters agree with the parameters of the affine Hecke algebra for a unipotent type by a twisted group algebra, and $\Mod\cH^{\fs}$ denotes the category of right unital $\cH^{\fs}$-modules.
\end{conjecture}

We note that the conjecture was originally stated for the Hecke algebras constructed by Solleveld in \cite{MR4432237}, where Solleveld constructed a Hecke algebra in a different way from using the theory of types.
Here, we consider the similar statement for the Hecke algebras attached to types.
In the case of Solleveld's Hecke algebras, he reduced Lusztig's conjecture \cite[\S1(a)]{2020arXiv200608535L} to the case of absolutely simple groups over $p$-adic fields, and proved it for most of these groups in \cite{MR4847675}.
For a comparison of the Hecke algebras of Solleveld's and ours, see \cite{MR4775177}.
In particular, the parameters of Solleveld's Hecke algebra coincide with those of ours (see \cite[Corollary~7.21]{MR4775177}).

The parameters of the affine Hecke algebras for unipotent types are explicitly described in \cite{Lusztig95unipotent, Lusztig02unipotentII} (see also \cite[Table\,1]{MR4847675}).

\begin{theorem}
Suppose that $G$ splits over a tamely ramified extension of $F$, and $p$ does not divide the order of the absolute Weyl group of $G$.
Then Conjecture~\ref{conjecturebyLusztig} holds.
\end{theorem}

\begin{proof}
According to the assumption, \cite[Theorem~7.12]{Fi-exhaustion}, and \cite[(4.3) Theorem (ii)]{BK-types}, there exists a $G$-datum $\Sigma = \bigl(
(\overrightarrow{G}, M^0), \overrightarrow{r}, (x_{0}, \{\iota\}), (K_{M^0}, \rho_{M^0}), \overrightarrow{\phi}
\bigr)$ such that $\Rep^\fs(G(F)) \simeq \Mod\cH(G(F), (K_{x_0}, \rho_{x_0}))$, where $(K_{x_0}, \rho_{x_0})$ denotes the $\fs$-type attached to $\Sigma$ by \cite[Section~4.1]{2024arXiv240807805A}.
Then the claim follows from Theorem~\ref{mainthm+AFMO}.
\end{proof}


\begin{thebibliography}{AFMO24b}
\addcontentsline{toc}{section}{References}

\bibitem[AFMO24a]{2024arXiv240807801A}
Jeffrey~D. {Adler}, Jessica {Fintzen}, Manish {Mishra}, and Kazuma {Ohara}, \emph{{Structure of Hecke algebras arising from types}}, arXiv
  e-prints (2024), arXiv:2408.07801.
  
  \bibitem[AFMO24b]{2024arXiv240807805A}
\bysame,
  \emph{{Reduction to depth zero for tame p-adic groups via Hecke algebra
  isomorphisms}}, arXiv e-prints (2024), arXiv:2408.07805.

\bibitem[AMS1]{MR3845761}
Anne-Marie Aubert, Ahmed Moussaoui, and Maarten Solleveld,
  \emph{Generalizations of the {S}pringer correspondence and cuspidal
  {L}anglands parameters}, Manuscripta Math. \textbf{157} (2018), no.~1-2,
  121--192. \MR{3845761}

\bibitem[AMS2]{MR3969871}
\bysame, \emph{Graded {H}ecke algebras for disconnected reductive groups},
  Geometric aspects of the trace formula, Simons Symp., Springer, Cham, 2018,
  pp.~23--84. \MR{3969871}
  
 \bibitem[AMS3]{2017arXiv170103593A}
\bysame, \emph{{Affine
  Hecke algebras for Langlands parameters}}, arXiv e-prints (2017),
  arXiv:1701.03593.

\bibitem[Ber84]{MR771671}
J.~N. Bernstein, \emph{Le ``centre'' de {B}ernstein}, Representations of
  reductive groups over a local field, Travaux en Cours, Hermann, Paris, 1984,
  Edited by P. Deligne, pp.~1--32. \MR{771671}

\bibitem[BK93]{MR1204652}
Colin~J. Bushnell and Philip~C. Kutzko, \emph{The admissible dual of {${\rm
  GL}(N)$} via compact open subgroups}, Annals of Mathematics Studies, vol.
  129, Princeton University Press, Princeton, NJ, 1993. \MR{1204652}

\bibitem[BK98]{BK-types}
\bysame, \emph{Smooth representations of reductive {$p$}-adic groups: structure
  theory via types}, Proc. London Math. Soc. (3) \textbf{77} (1998), no.~3,
  582--634. \MR{1643417}

\bibitem[Bou68]{MR0240238}
N.~Bourbaki, \emph{\'{E}l\'{e}ments de math\'{e}matique. {F}asc. {XXXIV}.
  {G}roupes et alg\`ebres de {L}ie. {C}hapitre {IV}: {G}roupes de {C}oxeter et
  syst\`emes de {T}its. {C}hapitre {V}: {G}roupes engendr\'{e}s par des
  r\'{e}flexions. {C}hapitre {VI}: syst\`emes de racines}, Actualit\'{e}s
  Scientifiques et Industrielles [Current Scientific and Industrial Topics],
  No. 1337, Hermann, Paris, 1968. \MR{0240238}

\bibitem[BT72]{MR327923}
F.~Bruhat and J.~Tits, \emph{Groupes r\'{e}ductifs sur un corps local}, Inst.
  Hautes \'{E}tudes Sci. Publ. Math. (1972), no.~41, 5--251. \MR{327923}

\bibitem[BZCHN24]{BZCHN}
David Ben-Zvi, Harrison Chen, David Helm, and David Nadler, \emph{Coherent
  {Springer} theory and the categorical {Deligne}-{Langlands} correspondence},
  Invent. Math. \textbf{235} (2024), no.~2, 255--344 (English).

\bibitem[DeB06]{MR2214792}
Stephen DeBacker, \emph{Parameterizing conjugacy classes of maximal unramified
  tori via {B}ruhat-{T}its theory}, Michigan Math. J. \textbf{54} (2006),
  no.~1, 157--178. \MR{2214792}

\bibitem[DL76]{MR393266}
P.~Deligne and G.~Lusztig, \emph{Representations of reductive groups over
  finite fields}, Ann. of Math. (2) \textbf{103} (1976), no.~1, 103--161.
  \MR{393266}

\bibitem[DM90]{MR1051245}
Fran\c~cois Digne and Jean Michel, \emph{On {L}usztig's parametrization of
  characters of finite groups of {L}ie type}, Ast\'erisque (1990), no.~181-182,
  6, 113--156. \MR{1051245}

\bibitem[DR09]{MR2480618}
Stephen DeBacker and Mark Reeder, \emph{Depth-zero supercuspidal {$L$}-packets
  and their stability}, Ann. of Math. (2) \textbf{169} (2009), no.~3, 795--901.
  \MR{2480618}

\bibitem[{Ete}24]{2024arXiv241213323E}
Arnaud {Eteve}, \emph{{Free monodromic Hecke categories and their categorical
  traces}}, arXiv e-prints (2024), arXiv:2412.13323.

\bibitem[{Ete}25]{2025arXiv250104113E}
\bysame, \emph{{Applications of the trace formalism to Deligne-Lusztig
  theory}}, arXiv e-prints (2025), arXiv:2501.04113.

\bibitem[Fin21a]{MR4357723}
Jessica Fintzen, \emph{On the construction of tame supercuspidal
  representations}, Compos. Math. \textbf{157} (2021), no.~12, 2733--2746.
  \MR{4357723}

\bibitem[Fin21b]{Fi-exhaustion}
\bysame, \emph{Types for tame {$p$}-adic groups}, Ann. of Math. (2)
  \textbf{193} (2021), no.~1, 303--346. \MR{4199732}

\bibitem[FKS23]{FKS}
Jessica Fintzen, Tasho Kaletha, and Loren Spice, \emph{A twisted {Y}u
  construction, {H}arish-{C}handra characters, and endoscopy}, Duke Math. J.
  \textbf{172} (2023), no.~12, 2241--2301. \MR{4654051}

\bibitem[FOS20]{MR4167790}
Yongqi Feng, Eric Opdam, and Maarten Solleveld, \emph{Supercuspidal unipotent
  representations: {L}-packets and formal degrees}, J. \'Ec. polytech. Math.
  \textbf{7} (2020), 1133--1193. \MR{4167790}

\bibitem[FOS22]{MR4515286}
\bysame, \emph{On formal degrees of unipotent representations}, J. Inst. Math.
  Jussieu \textbf{21} (2022), no.~6, 1947--1999. \MR{4515286}

\bibitem[GM20]{MR4211779}
Meinolf Geck and Gunter Malle, \emph{The character theory of finite groups of
  {L}ie type}, Cambridge Studies in Advanced Mathematics, vol. 187, Cambridge
  University Press, Cambridge, 2020, A guided tour. \MR{4211779}

\bibitem[GR02]{MR1901371}
David Goldberg and Alan Roche, \emph{Types in {${\rm SL}_n$}}, Proc. London
  Math. Soc. (3) \textbf{85} (2002), no.~1, 119--138. \MR{1901371}

\bibitem[GR05]{GoldbergRoche-Hecke}
\bysame, \emph{Hecke algebras and {${\rm SL}_n$}-types}, Proc. London Math.
  Soc. (3) \textbf{90} (2005), no.~1, 87--131. \MR{2107039}

\bibitem[HC70]{MR457579}
Harish-Chandra, \emph{Eisenstein series over finite fields}, Functional
  {A}nalysis and {R}elated {F}ields ({P}roc. {C}onf. for {M}. {S}tone, {U}niv.
  {C}hicago, {C}hicago, {I}ll., 1968), Springer, New York-Berlin, 1970,
  pp.~76--88. \MR{457579}

\bibitem[HL80]{MR570873}
R.~B. Howlett and G.~I. Lehrer, \emph{Induced cuspidal representations and
  generalised {H}ecke rings}, Invent. Math. \textbf{58} (1980), no.~1, 37--64.
  \MR{570873}

\bibitem[Kal19a]{kal19}
Tasho Kaletha, \emph{Regular supercuspidal representations}, J. Amer. Math.
  Soc. \textbf{32} (2019), no.~4, 1071--1170. \MR{4013740}

\bibitem[Kal19b]{2019arXiv191203274K}
\bysame, \emph{{Supercuspidal L-packets}}, arXiv e-prints (2019),
  arXiv:1912.03274.

\bibitem[KP23]{KalethaPrasad}
Tasho Kaletha and Gopal Prasad, \emph{Bruhat-{T}its theory---a new approach},
  New Mathematical Monographs, vol.~44, Cambridge University Press, Cambridge,
  2023. \MR{4520154}

\bibitem[KY17]{Kim-Yu}
Ju-Lee Kim and Jiu-Kang Yu, \emph{Construction of tame types}, Representation
  theory, number theory, and invariant theory, Progr. Math., vol. 323,
  Birkh\"{a}user/Springer, Cham, 2017, pp.~337--357. \MR{3753917}

\bibitem[Lus84]{MR742472}
George Lusztig, \emph{Characters of reductive groups over a finite field},
  Annals of Mathematics Studies, vol. 107, Princeton University Press,
  Princeton, NJ, 1984. \MR{742472}

\bibitem[Lus88]{MR1021495}
\bysame, \emph{On the representations of reductive groups with disconnected
  centre}, no. 168, 1988, Orbites unipotentes et repr\'esentations, I, pp.~10,
  157--166. \MR{1021495}

\bibitem[Lus95]{Lusztig95unipotent}
\bysame, \emph{Classification of unipotent representations of simple
  {$p$}-adic groups}, Internat. Math. Res. Notices (1995), no.~11, 517--589.
  \MR{1369407}

\bibitem[Lus02]{Lusztig02unipotentII}
\bysame, \emph{Classification of unipotent representations of simple
  {{\(p\)}}-adic groups. {II}}, Represent. Theory \textbf{6} (2002), 243--289
  (English).

\bibitem[Lus15]{MR3416310}
\bysame, \emph{Unipotent representations as a categorical centre}, Represent.
  Theory \textbf{19} (2015), 211--235. \MR{3416310}

\bibitem[Lus17]{MR3701897}
\bysame, \emph{Non-unipotent representations and categorical centres}, Bull.
  Inst. Math. Acad. Sin. (N.S.) \textbf{12} (2017), no.~3, 205--296.
  \MR{3701897}

\bibitem[{Lus}20]{2020arXiv200608535L}
\bysame, \emph{{Open problems on Iwahori-Hecke algebras}}, arXiv e-prints
  (2020), arXiv:2006.08535.

\bibitem[LY20]{MR4108915}
George Lusztig and Zhiwei Yun, \emph{Endoscopy for {H}ecke categories,
  character sheaves and representations}, Forum Math. Pi \textbf{8} (2020),
  e12, 93. \MR{4108915}

\bibitem[Mor93]{Morris}
Lawrence Morris, \emph{Tamely ramified intertwining algebras}, Invent. Math.
  \textbf{114} (1993), no.~1, 1--54. \MR{1235019}

\bibitem[Mor99]{MR1713308}
\bysame, \emph{Level zero {$\bf G$}-types}, Compositio Math. \textbf{118}
  (1999), no.~2, 135--157. \MR{1713308}

\bibitem[MP94]{MR1253198}
Allen Moy and Gopal Prasad, \emph{Unrefined minimal {$K$}-types for {$p$}-adic
  groups}, Invent. Math. \textbf{116} (1994), no.~1-3, 393--408. \MR{1253198}

\bibitem[MP96]{MR1371680}
\bysame, \emph{Jacquet functors and unrefined minimal {$K$}-types}, Comment.
  Math. Helv. \textbf{71} (1996), no.~1, 98--121. \MR{1371680}

\bibitem[MS14]{MR3157998}
Michitaka Miyauchi and Shaun Stevens, \emph{Semisimple types for {$p$}-adic
  classical groups}, Math. Ann. \textbf{358} (2014), no.~1-2, 257--288.
  \MR{3157998}

\bibitem[Oha24]{MR4775177}
Kazuma Ohara, \emph{A comparison of endomorphism algebras}, J. Algebra
  \textbf{659} (2024), 183--343. \MR{4775177}

\bibitem[Roc98]{MR1621409}
Alan Roche, \emph{Types and {H}ecke algebras for principal series
  representations of split reductive {$p$}-adic groups}, Ann. Sci. \'{E}cole
  Norm. Sup. (4) \textbf{31} (1998), no.~3, 361--413. \MR{1621409}

\bibitem[Sol22]{MR4432237}
Maarten Solleveld, \emph{Endomorphism algebras and {H}ecke algebras for
  reductive {$p$}-adic groups}, J. Algebra \textbf{606} (2022), 371--470.
  \MR{4432237}

\bibitem[Sol23a]{Solleveld23unipotentLLC}
\bysame, \emph{A local {Langlands} correspondence for unipotent
  representations}, Am. J. Math. \textbf{145} (2023), no.~3, 673--719
  (English).

\bibitem[Sol23b]{MR4620884}
\bysame, \emph{On unipotent representations of ramified {$p$}-adic groups},
  Represent. Theory \textbf{27} (2023), 669--716. \MR{4620884}

\bibitem[Sol25]{MR4847675}
\bysame, \emph{Parameters of {H}ecke algebras for {B}ernstein components of
  {$p$}-adic groups}, Indag. Math. (N.S.) \textbf{36} (2025), no.~1, 124--170.
  \MR{4847675}

\bibitem[SS08]{secherre-stevens:glnd-4}
V.~S\'{e}cherre and S.~Stevens, \emph{Repr\'{e}sentations lisses de {${\rm
  GL}_m(D)$}. {IV}. {R}epr\'{e}sentations supercuspidales}, J. Inst. Math.
  Jussieu \textbf{7} (2008), no.~3, 527--574. \MR{2427423}

\bibitem[SX24]{2024arXiv241119846S}
Maarten {Solleveld} and Yujie {Xu}, \emph{{Hecke algebras and local Langlands
  correspondence for non-singular depth-zero representations}}, arXiv e-prints
  (2024), arXiv:2411.19846.

\bibitem[Tay19]{MR3935811}
Jay Taylor, \emph{The structure of root data and smooth regular embeddings of
  reductive groups}, Proc. Edinb. Math. Soc. (2) \textbf{62} (2019), no.~2,
  523--552. \MR{3935811}

\bibitem[Yu01]{Yu}
Jiu-Kang Yu, \emph{Construction of tame supercuspidal representations}, J.
  Amer. Math. Soc. \textbf{14} (2001), no.~3, 579--622. \MR{1824988}

\end{thebibliography}

\providecommand{\bysame}{\leavevmode\hbox to3em{\hrulefill}\thinspace}
\providecommand{\MR}{\relax\ifhmode\unskip\space\fi MR }
\providecommand{\MRhref}[2]{%
  \href{http://www.ams.org/mathscinet-getitem?mr=#1}{#2}
}
\providecommand{\href}[2]{#2}

\end{document}